\documentclass{article}

\usepackage{amsmath,amssymb}
\usepackage{amsthm}
\usepackage{amsrefs}
\usepackage{graphicx}
\usepackage{hyperref}
\usepackage{xcolor}
\usepackage{fourier}
\usepackage{geometry}
\definecolor{antiquebrass}{rgb}{0.8, 0.58, 0.46}
\usepackage{natbib}
\usepackage{authblk}

\newcommand{\bs}{\boldsymbol}
\newtheorem{theorem}{Theorem}

\newtheorem{lemma}{Lemma}
\newcommand{\bl}{\big\langle}
\newcommand{\br}{\big\rangle}

\def\cal{\mathcal}

\def\wh{\widehat}
\def\wt{\widetilde}
\def\blue{\color{black}}
\def\red{\color{black}}
\def\brown{\color{black}}
\usepackage{ulem}

\title{Coarse-graining Langevin dynamics using   reduced-order techniques}
\author[1]{Lina Ma}
\affil[1]{Department of Mathematics, Trinity College, Hartford, CT 06106, USA.}%
\author[2]{ Xiantao Li}
\affil[2]{Department of Mathematics, the Pennsylvania State University, University Park, PA 16802-6400, USA.}%
\author[3]{Chun Liu}
\affil[3]{Department of Applied Mathematics, Illinois Institute of Technology, Chicago, IL 60616, USA}

\begin{document}
\maketitle
\begin{abstract}
This paper considers the reduction of the Langevin equation arising from bio-molecular models. {\brown To facilitate the construction and implementation of the reduced models,}  the problem is formulated as a {reduced-order modeling} problem. {\brown The reduced models can then be directly obtained from a Galerkin projection to appropriately defined Krylov subspaces}.  The equivalence to { a} moment-matching procedure, {\brown previously implemented in \cite{ma2016derivation},}   is proved.  A particular emphasis is placed on 
the reduction of the stochastic noise, {\brown which is absent in many order-reduction problems}. In particular, for order less than six we can show the reduced model obtained from the subspace projection automatically satisfies the fluctuation-dissipation theorem.
Details for the implementations, including a {\brown bi-orthogonalization procedure} and the minimization of the number of matrix multiplications, will be discussed as well.
\end{abstract}
\maketitle

\section{Introduction}

Langevin dynamics models arise from a wide variety of problems, especially where a mechanical system is subject to random forces that can be modeled by white noise, e.g.,  {\brown as in} \cite{Schlick2002}. 
A practical issue arises when the dimension of system is large, in which the computational cost can be overwhelming. For example, in bio-molecular models, the 
degrees of freedom are associated with the position and momentum of the constituting atoms, {\brown and the large dimensionality makes it difficult to probe large-scale biological processes over an extended period of time.}   In this case, it is of great  interest to develop reduced models, which in bio-molecular modeling, is known as coarse-graining \cite{Leach01, espanol2004statistical, riniker2012developing, noid2013perspective,voth2008coarse}. 

There are multiple benefits from such an approach. For example, reduced models can capture directly the dynamics of certain quantities of interest. Secondly, with the reduction of the dimension, the computational cost can be reduced dramatically. { In addition}, the quantities of interest often correspond to slow variables. By eliminating fast variables, the time step can also be increased considerably. This allows one to access longer time scales \cite{riniker2012developing}. 

There has been {\brown tremendous}  recent progress in the development of coarse-grained models \cite{curtarolo_dynamics_2002, IzVo06, kauzlaric2011three, kauzlaric2011bottom, lange2006collective, Li2009c,oliva2000generalized,stepanova_dynamics_2007,LiXian2014,li2015incorporation}. Most effort, however, is thermodynamics based. Namely, one aims to construct the free energy associated with the reduced variables, which then yields the driving force for the reduced dynamics, known as the potential of mean forces (PMF) \cite{marrink2007martini,monticelli2008martini}. As pointed out in \cite{marrink2007martini,monticelli2008martini}, the damping mechanics, which also plays an important role in the reduced dynamics, is not  part of the construction. 

In this work, we are interested in an {\it equation-based} derivation, where the reduced model can be derived directly from the Langevin dynamics.  {\red Deriving reduced models from a  stochastic dynamical system has been a subject of extensive studies, the most well known of which is the homogenization approach \cite{pavliotis2008multiscale}. Another important approach is to employ a coordinate transformation using normal forms to separate out the degrees of freedom that are less relevant \cite{roberts2008normal}.  More recently, Legoll and Lelievre proposed to use conditional expectations to derived reduced models \cite{legoll2010effective}.  Overall, these methods  require either significant scale separation assumption, or simple functions forms in the stochastic differential equations, which for bi-molecular models, does not apply.  For example, the force-field for biomolecular models typically involves complicated function forms. } 

{\red Meanwhile,  in the field  of molecular modeling there are also many methods that were proposed to coarse-grain a molecular dynamics model. Most of these methods are derived from a Hamiltonian system of ODEs \cite{li2014coarse,kauzlaric2011bottom,IzVo06,LiE07,fricks2009time,Darve_PNAS_2009,curtarolo_dynamics_2002,ChSt05}, either motivated by or directly  obtained, from the Mori-Zwanzig projection formalism \cite{Mori1965b,Zwanzig73}.  Strictly speaking, such a procedure will break down for stochastic models, due to the absence of the semi-group evolution operator.}
 For Langevin dynamics, one empirical coarse-graining approach is the partition method  \cite{sweet2008normal}, in which the variables are projected into appropriate subspaces.  However, the approach proposed in  \cite{sweet2008normal} does not reduced the number of variables. Rather, it is a numerical integration algorithm.   The main reduction comes from filtering out high frequency modes in the numerical algorithm. In our previous work \cite{ma2016derivation}, we have furthered this approach, by eliminating the fast-variables. This gives rise to a generalized Langevin equation (GLE) for the reduced variables. In principle, the GLE, under proper assumptions, is an exact model. %
After this reduction of the spatial dimensions, a temporal reduction was introduced to represent the memory term with a small number of auxiliary variable. Known as Markovian embedding, this procedure approximates the GLEs by using an extended system of stochastic differential equations (SDEs) with white noise.  The main idea is using a rational approximation for the Laplace transform, and the coefficients are determined based on a Hermite interpolation. The important advantage is that the approximation can be written as an extended system of SDE with no memory. 

{\brown A  well known issue in Pad\`{e} type of approximations}
is that when more conditions are incorporated, the resulting models tend to be ill-posed. In particular, the coefficient matrices are usually ill-conditioned, making it impractical.  {\brown Therefore, an important} focus of this paper is on  re-formulating the coarse-graining procedure into a {reduced-order problem}, which has been widely studied \cite{bai2002krylov,feldmann1995reduced}. 
In particular, we observe a feedback loop between the coarse-grain variables and the additional degrees of freedom, i.e., the fast variables. More specifically, the slow variables impose a mechanical force on the fast dynamics, and in turn, such influences will be propagated back as a force on the slow variables. As a result, the elimination of fast variables can be viewed as an order reduction problem, in that it is a large-dimensional dynamical system with low-dimensional input and low-dimensional output.  {\brown We will show that with an appropriate reformulation of the fast dynamics, } the transfer function from the order-reduction problem corresponds precisely to the memory kernel in the GLE. For such problems, one robust numerical method is the Krylov subspace projection \cite{bai2005reduced,bai2002krylov}, which uses a Galerkin projection onto Krylov subspaces. The subspaces can be orthogonalized using the Lanczos algorithm \cite{feldmann1995reduced,loher2006reliable}.  {\brown As a result, instead of mannually constructing the auxiliary system on a case-by-case basis as in the moment matching approach \cite{ma2016derivation}, we can automate the procedure numerically.
More importantly, the bi-orthogonalization alleviate the problem of having ill-conditioned matrices.}
 
For the current problem, the presence of the noise presents another critical issue. Namely, the random noise in the GLE must satisfy the second fluctuation-dissipation theorem (FDT) \cite{Kubo66}, {\brown a necessary condition for the solution of the GLE to be stationary and to have the correct variance.} In the Galerkin projection method, both the noise and the kernel function are being approximated. {\brown In general,  they do not satisfy the second FDT, unless the subspaces are properly selected.  We will provide two conditions that ensure such consistency, and we will show Krylov subspaces that fullfill these conditions. }

This paper is organized as follows. Section 2  describes the derivation the GLE system.   The classical approach of  approximating the  Laplace transform of the  memory kernel function $\theta(t)$ with a rational function will be presented. 
 Section 3 presents a formulation using the Galerkin  projection to general subspaces. Criteria will be provided in order to maintain the FDT in the reduced system. In Section 4, we introduce appropriate  Krylov subspaces to fulfill the criteria.  The resulting system will also be compared to a moment-matching procedure and the equivalence is proved in this section. Section 5 addresses two important issues in the numerical implementation.  Numerical examples { are} shown in Section 6.  

\section{Mathematical Derivation}
\subsection {The Reduction of the Full Langevin Dynamics Model}
We start with the full Langevin dynamics model with $N$ atoms. After proper mass scaling \cite{Schlick2002},  the system can be expressed as follows,
\begin{equation}\label{scaleLGE}
\left\{
 \begin{aligned}
   \dot{x}(t)= & v(t),\\
     \dot{v}(t)=& F(x) - \Gamma v(t) + f(t),   
 \end{aligned}
 \right.
\end{equation}
where $x=(x_1,x_2,\dots,x_N)$ denotes the displacement of all the atoms,  $F(x)$ is the force derived from an empirical potentials $V(x)$ with $F=-\nabla V$, $\Gamma$ denotes the damping coefficient for the friction term with dimension $\mathbb{R}^{3N\times 3N}$, and $f(t)$ is a stochastic force, usually modeled by a Gaussian white noise, which satisfies the fluctuation-dissipation theorem (FDT),
\begin{equation}
\bl f(t), f(t')^\intercal \br=2k_B T\Gamma \delta(t-t').
\end{equation}
{\brown For example, the random force can be written in the conventional form: $df(t)= \sigma dW(t)$ with $W(t)$ being the standard Brownian motion, and $\sigma\sigma^\intercal= 2k_B T\Gamma.$}
Here, $k_B$ is the Boltzmann constant, and $T$ is the temperature of the system. This FDT is crucial to ensure that the system reaches the correct equilibrium state  \cite{Kubo66} . 


Implementing the full Langevin dynamics model can be very expensive, due to the large number of atoms involved in the entire system. 
Here we briefly go over a reduction procedure. More details can be found in   \cite{ma2016derivation}.


The first step in the reduction procedure is to identify slow variables, which at the same time, are sufficient to describe the overall dynamics. In principle, these variables can be selected 
by transforming the system into normal forms  \cite{roberts2008normal}.  For bio-molecules, a more intuitive and more efficient approach is based on the residues,  the building blocks of proteins, by choosing the center of mass of each amino acid.  Mathematically, this can be expressed as a small number of basis functions \cite{TaGaMaSa00}, which span a subspace, denoted here by $Y$, with its orthogonal complement denoted by  $Y^\perp$.      $Y$ has dimension $m$ and $Y^\perp$ has dimension ${3N-m}$: $m\ll 3N.$  We denote the basis vectors by \(\{\phi_i\}\) and  \(\{\psi_i\}\), respectively,  as follows,
$$
Y={\rm span} \{\phi_1,\phi_2,\dots,\phi_m\}, \quad
Y^{\perp} = {\rm span} \{\psi_1,\psi_2,\dots,\psi_{3N-m}\}.
$$
Taking these basis vectors as columns and forming matrices $\Phi$ and $\Psi$,  one can  decompose the solution $x$ in the following form,
\begin{equation}\label{eq: qxi}
 x(t)=\Phi q(t) + \Psi \xi(t),
\end{equation}
where $q \in \mathbb{R}^m$ and $ \xi \in \mathbb{R}^{3N-m} $ are nodal values associated with the basis vectors. Similarly, 
$$
v(t)=\Phi p(t) + \Psi \eta(t).
$$ 
Meanwhile, a linearization of the force $F\approx -Ax$ is considered, e.g., by principal component analysis (PCA)  \cite{stepanova_dynamics_2007}:
$$
\bl x(t), x(t)^\intercal  \br = k_BT A^{-1},
$$
which ensures that the covariance of the displacement is correct.

Now define the following projected matrices and vectors, 
$$ A_{11} = \Phi^\intercal  A \Phi, \quad A_{12} = \Phi^\intercal  A \Psi, \quad \Gamma_{11} = \Phi^\intercal \Gamma\Phi, \quad \Gamma_{12} = \Phi^\intercal  \Gamma \Psi,   \quad f_1=\Phi^\intercal  f(t),
$$
$$
 A_{21}=\Psi^\intercal  A \Phi, \quad A_{22}=\Psi^\intercal A \Psi, \quad \Gamma_{21} = \Psi^\intercal \Gamma \Phi, \quad \Gamma_{22}=\Psi^\intercal \Gamma\Psi, { \quad f_2=\Psi^\intercal  f(t)}.
$$

By using this partition of variables, the original Langevin dynamics can be written in terms of the following first order stochastic differential equations (SDEs),
\begin{align}
 \left\{
 \begin{aligned}\label{eq: lg'}
   \dot{q}(t)= &  p(t), \\
   \dot{p}(t)=&\Phi^\intercal  F(\Phi q) - A_{12} \xi(t)  - \Gamma_{11} p(t) - \Gamma_{12}  \eta (t)+{f_1(t)},
   \end{aligned}
    \right.\\
   \left\{
   \begin{aligned}\label{eq: lg'2}
   \dot{\xi}(t)= &  \eta(t), \\   
   \dot{\eta}(t)= & -A_{21} q(t)  - A_{22}\xi(t) - \Gamma_{21} p(t) - \Gamma_{22}  \eta(t) + { f_2(t)}.
   \end{aligned}
    \right.
\end{align}
{\red The linearization of the high-frequency modes has been based on numerous observations, e.g., \cite{hayward2002temperature}.
Essentially, we assume that the low frequency can be well captured by the basis functions in $\Phi$, and the high frequency is nearly Gaussian. 
For example, in the rotation-translation block (RTB) approach, each residue is allowed to move as a rigid body. There is overwhelming evidence that the low-frequency normal modes are well represented by the subspace spanned by such basis functions \cite{TaGaMaSa00}.  
}

Here $(q, p)$ are the reduced/coarse-grained variables. Notice that the interactions involving the fast variables $\xi$ have been 
linearized. By eliminating  $(\xi, \eta)$, we have derived a low-dimensional reduced model {\cite{ma2016derivation}},
\begin{equation}\label{eq: GLE}
\left\{
\begin{aligned}
 \dot{q}(t)= &  p(t), \\
   \dot{p}(t)=& F_\text{eff}(q) - \Gamma_{11} p (t)- \int_0^t  \theta(t-\tau) p(\tau)d\tau + \wt {f}(t).
\end{aligned}
\right.
\end{equation}
The effective force for the reduced system is
\begin{equation}\label{eq: f-eff}
 F_\text{eff}(q)= \Phi^\intercal  F(\Phi q) -A_{12}A_{22}^{-1}A_{21}q.
\end{equation}
{
Compared to system (\ref{eq: lg'}), the force $F_{\rm eff}$ has an extra term $ -A_{12}A_{22}^{-1}A_{21}q$ from the derivation.
}
$\theta(t)$ is the {\it memory kernel function}, which is expressed in terms of a matrix exponential, 

\begin{equation}\label{eq: theta}
 \theta(t)= 
 \big[ A_{12}, \;\Gamma_{12}\big] e^{Dt} 
 \left[
\begin{array}{cc}
 A_{22}^{-1} &  0\\
 0 & -I
 \end{array}
 \right]
 \left[
\begin{array}{c}
  A_{21}  \\
  \Gamma_{21}
 \end{array}
\right],
\end{equation}
{ where the matrix $D \in \mathbb{R}^{(6N-2m)\times (6N-2m)}$  is defined as,}
\begin{equation}\label{eq: Gmat}
D=
\left[
\begin{array}{cc}
 0 & I\\
 -A_{22} &\;\;\;\; -\Gamma_{22} 
 \end{array}
 \right] .
\end{equation}

It has also been shown {\red in  \cite{ma2016derivation} 
\begin{equation}\label{eq: f-ran}
\wt{f} = { f_1(t)}
  - \big[ A_{12}, \;\Gamma_{12}\big]  
 \int_0^\intercal  e^{D(t-s)} \left[
\begin{array}{c}
  0 \\
{ f_2(s)}
 \end{array}
 \right] ds
 - \big[ A_{12}, \;\Gamma_{12}\big]
e^{D t} 
 \left[
\begin{array}{c}
 {\xi}(0)+A_{22}^{-1}A_{21}q(0)\\
\eta(0)
 \end{array}
 \right]
. 
\end{equation}
This random force } is a stationary Gaussian random process with mean zero, satisfying the second fluctuation-dissipation theorem:
\begin{equation}\label{eq: fdt3}
 \bl \wt {f}(t) \wt {f}(t')^\intercal \br = 2k_B T { \Gamma_{11}} \delta(t-t') + k_B T \theta(t-t').
\end{equation}

\bigskip
\bigskip
{
Equation \eqref{eq: GLE} is known as the generalized Langevin equation (GLE). Currently there are primarily  three existing methods to solve the GLE numerically.
The first approach is to directly approximate the memory term, either by using quadrature formula, or by approximating the kernel function with a sum of exponentials. Known as the Prony sum, the later approach replaces the memory integral by additional variables that can be updated using certain recurrence formulas or by solving an ODEs system \cite{arnold2003,Jiang2004}.  The random noise can be approximated by introducing noises in those ODEs \cite{baczewski2013numerical}. However, the approximation of the sum of exponentials requires the values of the kernel function \eqref{eq: theta}, which is difficult to compute due to the large dimensionality of the matrix $D$ in the matrix exponential. The second approach is to eliminate the memory effect by approximating the kernel function with a delta function in time \cite{hijon2006markovian,kauzlaric2012markovian}. This approximation can be quite effective when the memory effect is not strong. But in general, the accuracy is quite limited.  The third approach is to approximate the memory effect by introducing auxiliary variables. This has been motivated by the Mori's continued-fraction approach \cite{Mori1965b},  and has been pursued by many groups \cite{LiXian2014,li2015incorporation,Darve_PNAS_2009,li2014coarse,ma2016derivation}.
}


 For example, {\brown in \cite{ma2016derivation}}, the  first order approximation leads to an extended dynamics with auxiliary variable $z$,
\begin{equation}\label{eq: 1st-order}
\left\{
\begin{aligned}
 \dot{q}(t)= & p(t), \\
   \dot{p}(t)=&F_\text{eff}(q)- \Gamma_{11} p(t) - z(t) + { f_1}(t),\\
   \dot{z}(t)=&  Bz(t) +  C{p}(t) + { \zeta(t)}.
\end{aligned}
\right.
\end{equation}
The coefficients $ B$ and $ C$ can be found by using a  `moment matching' procedure, and we will elaborate on such procedures in section \ref{sec: momentmatching}. At the same time, methods have been established to sample the additive noise $\zeta(t)$  to ensure the FDT \eqref{eq: fdt3}. 

In theory, it is possible to advance to high order approximations using the above methods, e.g., a third order method  \cite{ma2016derivation}. However, in practice, the matrices generated from the moment matching procedure tend to become ill-conditioned as the order of approximation increases. Moreover, the covariance of the noise and the covariance of the auxiliary variable $z$  need to be constructed specifically  for  each  order of approximation to ensure the FDT \eqref{eq: fdt3}, which is nontrivial. Therefore, it is  important to develop an alternative method to improve the robustness and automate the procedure. Inspired by order reduction methods for large-scale dynamical system, we will formulate the current problem as an order reduction problem with stochastic noise. The key  is to identify the low-dimension input and low-dimension output.

\section{Model Reduction for the  Stochastic Model}

\subsection{A Reformulation of the Orthogonal Dynamics}
We will first introduce vector and matrix notations to rewrite the system \eqref{eq: lg'2} in a more compact form.
Let $ y=(\xi, \eta)^\intercal$  represents   the partitioned variables, and $ u(t) = (q,p)^\intercal$  represents the coarse grained variables. {\brown System (\ref{eq: lg'2})  can be rewritten as  }  
the following  { SDEs}:
\begin{equation}\label{eq: orth0}
   \dot{ y}(t)= D  y (t)+  {\wh{R}}  u(t) +{g}(t), \quad  y(0) \sim \cal{N}(0,k_B T Q).
\end{equation}
with 
\begin{equation} 
{\wh{R}}=
 \left[
 \begin{array}{cc}
   0 & 0 \\
  -A_{21} &  -\Gamma_{21}
  \end{array}\right],
  \quad g=\left[
 \begin{array}{c}
   0\\
   f_2(t)
  \end{array}\right]. 
\end{equation}

The matrix $Q$ determines the initial covariance of $ y$, given by,
\begin{equation}
Q=
\left[
\begin{array}{cc}
 A_{22}^{-1} &  \;\;\;\;0\\
 0 & I
 \end{array}
 \right].
\end{equation}
Further we let $\Sigma$ be the variance of the {\brown Gaussian} noise $g(t)$. It follows the Lyapunov equation, to ensure the stationarity of the solution,
  \begin{equation}\label{eq: Lya}
  \Sigma = -k_BT \big( DQ + Q D^\intercal \big).
  \end{equation}
  It can be directly verified that,
 \begin{equation}
\Sigma=\left[
\begin{array}{cc}
 0 &  0\\
 0& 2 k_B T\Gamma_{2,2}
 \end{array}
 \right].
\end{equation}
{\brown At the same time, we define}
\begin{align}
 \quad L=\big[ A_{12}, \;\Gamma_{12}\big], \quad R=[A_{22}^{-1}A_{21}, -\Gamma_{21}]^\intercal.
\end{align}
{\brown Now the equation} (\ref{eq: lg'}) can be written as
\begin{align}
 \left\{
 \begin{aligned}\label{eq: cgGLE}
   \dot{q}(t)= &  p(t), \\
   \dot{p}(t)=&\Phi^\intercal  F(\Phi q)   - \Gamma_{11} p(t) - L\bs y+{f_1(t)}.
   \end{aligned}
    \right.
\end{align}

{\brown The corresponding memory kernel in \eqref{eq: GLE} is given by,}
\begin{equation}\label{eq: mem}
 \theta(t)=L e^{Dt} R.
\end{equation}

{\brown It is at this point that we recognize the similarity to an order reduction problem: The large-dimensional dynamics \eqref{eq: orth0} contains an input variable $u(t)$, which is low-dimensional. Moreover, of direct importance to the coarse-grained dynamics \eqref{eq: cgGLE} is $Ly$, which again is low-dimensional. Also observed, however, is that the dimensions of $L$ and $\wh{R}$ are different. Fortunately, we can reformulate the problem into the following equivalent dynamics (\ref{eq: GLEn3}), where the input and output dimensions are the same.
\begin{align}\label{eq: GLEn3}
\left\{
\begin{aligned}
\dot{q}(t) &=p(t)\\
\dot{p}(t) &=F_{\rm eff}(q)-\Gamma_{11} p(t)-L  y + f_1(t),\\
\dot{ y}(t) &= D  y(t) + Rp(t) + g(t),\quad    y(0) \sim \cal{N}(0,k_B T Q).
\end{aligned}
\right.
 \end{align}

 }




\begin{theorem}\label{thm1}
Consider the following dynamics:
\begin{equation}\label{eq: orth1}
\dot { y}_1(t)= D y_1(t) + Rp(t) +  {{g}}(t), \quad  y_1(0) \sim {\cal N}(0, k_B TQ).
\end{equation}
{\red With a substitution into the first two equations in \eqref{eq: GLEn3} (in which $y$ is replaced by $y_1$), one obtains a GLE that is equivalent to (\ref{eq: GLE}).
}


\end{theorem}
\begin{proof}

Using a variation of constant formula, we find,
\[
 y_1(t) = e^{D t}  y_1(0) + \int_0^t  e^{D(t-\tau)}Rp(\tau) d\tau
 +  \int_0^t  e^{D(t-\tau)}{g}(\tau)d\tau. \]
Next we define the out quantity $ w_1(t)$ from \eqref{eq: orth1},
\begin{align}
 w_1(t) = L  y_1 = \int_0^t \theta(t-\tau) p(\tau) d\tau+ Le^{Dt} y_1(0)+\int_0^t  Le^{D(t-\tau)}{g}(\tau) d\tau
=:  \int_0^t  \theta(t-\tau) p(\tau) d\tau+  \zeta(t).
\end{align}
Here $ \zeta$ is the sum of the last two terms. For $t>t'$, we have,
\begin{align}
\langle  \zeta(t) \zeta^\intercal(t')\rangle = 
k_BT Le^{Dt}Qe^{D^\intercal t'} L^\intercal+ \Big\langle \int_{0}^{t} \int_0^{t'}Le^{D(t-\tau)} g(\tau) g^\intercal(\tau') e^{D^\intercal (t-\tau')} L^\intercal d\tau' d\tau\Big\rangle 
=k_BTLe^{D(t-t')} QL^\intercal.
\end{align}
{\brown The second step can be carried out by using the It\^{o}'s isometry.}

Now we replace the term {$-A_{12}\xi(t) - \Gamma_{12}\eta (t)$} by $ w_1(t)$  in system \eqref{eq: GLEn3}. We  have,
\begin{equation}\label{eq: GLEn}
\left\{
\begin{aligned}
 \dot{q}(t)= &  p(t), \\
   \dot{p}(t)=& F_\text{eff}(q)  - \Gamma_{11} p(t) - \int_0^t  \theta(t-\tau) p(\tau) d\tau- \zeta(t) + {f_1(t)}.
\end{aligned}
\right.
\end{equation}
Let $\wt  f_1(t)=f_1(t)-\zeta(t)$. With the assumption that the initial data of $ y_1$ is uncorrelated with the noise term, we get, 
\begin{align}
\begin{split}
\langle \wt  f_1(t) \wt  f_1^\intercal(t')\rangle =&
2k_BT \Gamma_{11} \delta(t-t')
-2k_BTLe^{D(t-t')} 
\left[
\begin{array}{c}
0\\\Gamma_{21}
\end{array}
\right]
+k_BTLe^{D(t-t')}QL^\intercal
\\
=&2k_BT \Gamma_{11} \delta(t-t')
+k_BT \theta(t-t').
\end{split}
\end{align}
The last step requires that
\begin{align}\label{eq: fdtcond}
QL^\intercal -2\left[
\begin{array}{c}
0\\\Gamma_{21}
\end{array}
\right] = R,
\end{align}
which can be easily verified. {\brown Now, according to theory of Gaussian processes \cite{Doob44}, the processes $\wt f(t)$ and $\wt f_1(t) $ are equivalent. }

{\brown Finally,  the memory terms in  (\ref{eq: GLEn}) with (\ref{eq: GLE}) are the same,  the proof of equivalence is thus completed. }

\end{proof}
It is clear that the dynamics (\ref{eq: orth1}) is very similar to dynamics (\ref{eq: orth0}), with subtle modification: $p(t)$ instead of $ u(t)$ is involved in the system. {\brown More importantly,  in  (\ref{eq: orth1})   the input and the output of the dynamics  have the same dimension. } Our following discussion will be based on $ y_1$, and instead, we will denote this term as $ y$ due to the equivalence.

\smallskip

We now have formulated the problem as a reduce-order problem: The dynamics of $ y$ involves a large-dimensional dynamical system, in which the variable $ p(t)$ is acting as a control variable. Meanwhile, what is of interest to the coarse-grained dynamics is the quantity $L  y$. As a result, we have at hand a large dynamical system with low-dimensional input and a low-dimensional output.

\subsection{{Properties of General }Galerkin Projections}
{\brown A remarkable success in order reduction problems is the Galerkin projection method to appropriately defined subspaces  \cite{bai2002krylov,villemagne1987model}.
 Motivated by such success, }
we first consider a general Galerkin projection of the SDEs \eqref{eq: orth1},
\begin{align}\label{eq: orth2}
\dot { y}(t) = D  y(t) + Rp(t)+g(t).
\end{align}
More specifically, we  seek $\wh { y}(t)$ in the subspace $X_n=\text{span}\{{ V_1}, V_2, \dots V_n\}$, { with each basis having $m$ columns}. We denote the space of test functions by $\wt  X_n=\text{span}\{W_1, \dots, W_n\}$. Now the projection can be stated as follows: find $\wh { y}(t)\in X_n$, such that for any  {\brown $ \chi(t) \in \wt  X_n$},
$$
(\dot { \wh{ y}}(t) - D \wh{  y}(t) - Rp(t)-g(t), {\brown \chi(t)} )=0.
$$

To put it in  a matrix-vector form, let $V=[{ V_1}, V_2, \dots V_n]$ and $W=[W_1, W_2, \dots, W_n]$, and we choose the columns as the basis for the two subspaces. The approximate solution is written as, 
\begin{equation}\label{eq: z2y}
\widehat{ y}(t) = V  z(t),
\end{equation}
with $ z(t)$ being the nodal values. 
Then the Galerkin projection yields, 
\begin{align}
\widehat{M} \dot{ z}(t) = \widehat{D}   z (t)+ W^\intercal R  p(t) + W^\intercal g(t),
\end{align}
where we have defined, \begin{equation}\label{eq: MD}
\begin{aligned}
\widehat{M}=W^{\intercal} V, \quad
\widehat{D}=W^{\intercal}D V.
\end{aligned}
\end{equation}
With the assumption that $\wh M$ is nonsingular, we can write
\begin{equation}\label{eq: projD}
\dot{ z}(t) = \widehat{M}^{-1} \widehat{D}  z(t) + \widehat{M}^{-1} W^\intercal R p(t) +\widehat{ f}(t), 
\end{equation}
{\red where 
\begin{align}\label{eq: whf}
\widehat{\bs f}(t)= \widehat{M}^{-1} W^\intercal \bs f(t),
\end{align}}
 and its covariance matrix is given by,
\begin{equation}\label{eq: sig}
  \langle \wh f(t) \wh f(t')^\intercal \rangle = \wh \Sigma \delta(t-t'), \quad \widehat{\Sigma}=  \widehat{M}^{-1} W^\intercal  \Sigma W \widehat{M}^{-\intercal}. 
\end{equation}

With this reduction, we can now write down the reduced model  involving  the variables $(p, q, z)$,
\begin{subequations}
\begin{align}
\begin{split}
 \dot{q}(t)= &  p(t), \\
   \dot{p}(t)=& F_\text{eff}(q) - \Gamma_{11} p(t) - LV z(t) + {f_1(t)},
   \end{split}\\
\dot{ z}(t) = &\widehat{M}^{-1} \widehat{D}  z(t) + \widehat{M}^{-1} W^\intercal R p(t) +\widehat{ f}(t), 
\end{align}\label{eq: Krylov'2}
\end{subequations}

In contrast to the conventional order  reduction problems  \cite{bai2005reduced}, the current approach yields a noise term. {\brown Its presence } brings up an important issue: appropriate conditions are needed to ensure that the solution reaches  correct equilibrium, which will be addressed here. 

Due to ergodicity, the solution of the original SDE,  $ y(t),$ will evolve into a stationary process, and we expect the approximate solution to become a  stationary process as well. Assuming that the initial variance of $ z$ is $k_BT\widehat{Q}$,  that is, {\brown
\begin{equation}
\langle z(0) z(0)^\intercal \rangle = k_BT\widehat{Q},
\end{equation}}
then the stationarity implies that $\widehat{Q}$ must satisfy the Lyapunov equation   \cite{risken1984fokker},
\begin{equation}\label{eq: cond-1}
  k_BT(\widehat{M}^{-1} \widehat{D} \widehat{Q} + \widehat{Q}\widehat{D}^\intercal \widehat{M}^{-\intercal}) = - \widehat{\Sigma}. \tag{Condition {\bf A}}
\end{equation}
This condition, as one of the necessary conditions to ensure the second FDT, will be referred to as {Condition {\bf A}}.

{\brown Meanwhile, the projected dynamics \eqref{eq: Krylov'2} corresponds to an approximation of the GLEs \eqref{eq: GLE}. 
This can be verified by directly solving \eqref{eq: orth2}, and then substitute $L\wh y$ into the equation for $p.$}
With direct calculations, we find that  the approximated kernel can be expressed as,
\begin{equation}\label{eq: th-approx}
 \theta(t) \approx  \wh{\theta}(t) :=L V e^{\widehat{M}^{-1} \widehat{D} t} \widehat{M}^{-1} W^\intercal R.
\end{equation}
Moreover, the low dimensional output is approximated by,
\begin{equation}
  w(t) \approx \widehat{ w}(t) = L\widehat{ y} = \int_0^T  \wh{\theta}(t-\tau) p(\tau) d\tau + \wh{\zeta}(t),  
\end{equation}
where
$$
\widehat{\zeta}(t) = LVe^{\widehat{M}^{-1}\widehat{D}t} z(0) + \int_0^T  LVe^{\widehat{M}^{-1}\widehat{D}(t-\tau)} \widehat{ f}(\tau)d\tau.
$$
{\brown
As a result, we obtain an approximate GLE model, 
\begin{equation}\label{eq: GLE'}
\left\{
\begin{aligned}
 \dot{q}(t)= &  p(t), \\
   \dot{p}(t)=& F_\text{eff}(q) - \Gamma_{11} p (t)- \int_0^T \wh\theta(t-\tau) p(\tau)d\tau + \widehat{\zeta}(t) + f_1(t).
\end{aligned}
\right.
\end{equation}
}

The term $\widehat{\zeta}(t)$  introduces an added Gaussian noise to the coarse-grained dynamics. 
 Together with the Lyapunov equation \eqref{eq: cond-1}, we can express its time  correlation as follows,
\begin{equation}\label{eq: thetat}
  \langle \wh{\zeta}(t) \wh{\zeta}(t')^\intercal \rangle = k_B T  L V e^{\widehat{M}^{-1}\widehat{D} (t-t')}  \wh{Q} V^\intercal L^\intercal=k_BT \wt\theta(t-t'),\quad \wt\theta(t):=  L V e^{\widehat{M}^{-1}\widehat{D} t}  \wh{Q} V^\intercal L^\intercal.
\end{equation}

\smallskip
Clearly, in general the correlation of the noise $\wt\theta(t)$ in \eqref{eq: thetat} might not be consistent with the memory kernel $\wh\theta(t)$  in \eqref{eq: GLE'} and \eqref{eq: th-approx}. Namely, the second FDT, a necessary condition for the reduced model to have the correct statistics, may not be fulfilled.   The following theorem identifies the condition under which such consistency can be guaranteed.

\begin{theorem}\label{thm:2fdt}
{\red The coarse grained dynamics \eqref{eq: Krylov'2} and \eqref{eq: GLE'} derived from the 
Petrov-Galerkin projection} will obey the second FDT \eqref{eq: fdt3}, if the following condition is satisfied:
$$
\widehat{M}\widehat{Q} V^\intercal L^\intercal=W^\intercal Q L^\intercal.
$$
\end{theorem}

\begin{proof}
Recall that $ w(t)=Ly$ from (35b) needs to be injected into the dynamics of the reduced variables (35a).
The resulting random noise  is  $\wt  \zeta=-\widehat{\zeta}(t) + f_1(t)$, with time correlation,
\begin{align}
\langle \wt  \zeta(t) \wt  \zeta^\intercal (t') \rangle=2k_BT \Gamma_{11}\delta(t-t') + 
 k_B T  L V e^{\widehat{M}^{-1}\widehat{D} (t-t')}  \wh{Q} V^\intercal L^\intercal
 -2k_BT LV e^{\widehat{M}^{-1}\widehat{D} (t-t')} \widehat{M}^{-1}{W}^\intercal 
\left[
\begin{array}{c}
0\\
\Gamma_{21}
\end{array}
\right].
\end{align}
It is clear that if,
\begin{align}\label{eq: fdtcond2}
\widehat{Q} V^\intercal L^\intercal - 2 \widehat{M}^{-1}W^\intercal 
\left[
\begin{array}{c}
0\\
\Gamma_{21}
\end{array}
\right] = \widehat{M}^{-1}W^\intercal R,
\end{align}
 this will result in the second FDT:
$$
\langle \wt {\zeta} (t) \wt {\zeta}^\intercal (t') \rangle = 2 k_BT\Gamma_{11}\delta (t-t')+ k_BT \widehat{\theta}(t-t').
$$
In light of Equation (\ref{eq: fdtcond}), Equation (\ref{eq: fdtcond2}) is equivalent to
\begin{align} \label{eq: th2}
\widehat{M}\widehat{Q} V^\intercal L^\intercal=W^\intercal Q L^\intercal.
\tag{Condition {\bf B}}
\end{align}

This equation will be referred to as condition {\bf B}.
\end{proof}

Conditions {\bf A} and {\bf B} constitute the basis for constructing consistent stochastic reduced models. While condition {\bf A} can be enforced by solving the Lyapunov equation,    condition {\bf B} may not be satisfied by an arbitrary Galerkin projection. Therefore, we need to choose appropriate subspaces  for this to hold automatically.

\section{The Projection to Krylov Subspaces}

In this section, we will construct Krylov subspaces for the Galerkin projection procedure, which subsequently leads to approximations of the memory kernel function and 
random force.  
{\brown 
We will also discuss several issues related to the practical implementations. }

It turns out that the Krylov subspace approach has a close connection to a two-point Pad\`{e} approximation, previously studied in   \cite{LiXian2014,li2014coarse,ma2016derivation,lei2016generalized} to incorporate both long time and short time statistics. We will review this approach briefly, which will be referred to as moment-matching, and then make connections to the Krylov subspace projection approach.
{ 
We will consider the case where the damping coefficient is constant, i.e., $\Gamma=\gamma I.$}

\subsection{The Moment Matching Approach}\label{sec: momentmatching}
Define the moments,
\begin{equation}\label{eq: moms}
 M_0= \theta(0), \;\;M_1= \theta'(0), \;\; \cdots, \;M_\ell=  \theta^{(\ell)}(0),\;\cdots, \;\;M_\infty=\int_0^\infty \theta(t)dt.
\end{equation}
Notice that the moment $M_\infty$ corresponds to the correlation time. 
With the moments, the memory function at $t=0$ can be expanded as:
$$
\theta(t)=M_0+M_1t+\frac{M_2}{2}t^2+\dots+\frac{M_\ell}{\ell!}t^\ell+\dots.
$$
Since the exact memory kernel is $Le^{Dt}R$, it is clear that the moments are given by,$$
M_0=LR, \quad M_1=LDR,\quad\dots \quad M_\ell=LD^\ell R,\quad M_{\infty} = - LD^{-1}R.
$$

Meanwhile, the Laplace transform can be expanded near zero, 
\begin{equation}\label{eq: Th}
\Theta(s)=\frac{M_0}{s}+\frac{M_1}{s^2} + \dots+\frac{M_\ell}{s^{\ell+1}}+\dots,
\end{equation}
which can be obtained by repeated integration by parts  \cite{BlHa86}.


{\brown The moment matching procedure is essentially a rational approximation}  for the Laplace transform of the memory kernel,
$$\Theta_n(s)=(s^nI - s^{n-1} B_0 - s^{n-2}B_1-\dots-B_{n-1})^{-1}(s^{n-1} C_0 + s^{n-2}C_1+\dots+C_{n-1}),$$
such that,
$$\Theta_n(0)=\Theta(0) (= M_\infty), \quad \theta_n^{(\ell)}(0)=\theta^{(\ell)}(0) (=M_\ell) \;{\rm for }\; \ i=0,\dots, 2n-2,$$
To solve for the coefficients $B_i$, one needs to solve a linear system,
\begin{align}\label{multiB}
\left[
\begin{array}{cccc}
-M_{\infty} &M_0&\dots&M_{n-2}\\
M_0&M_1&\dots&M_{n-1}\\
\dots\\
M_{n-2}&M_{n-1}&\dots&M_{2n-3}
\end{array}
\right]
\left[
\begin{array}{c}B_{n-1}\\B_{n-2}\\\dots\\B_0\end{array}
\right]=
\left[
\begin{array}{c}
M_{n-1}\\
M_n\\
\dots\\
M_{2n-2}
\end{array}
\right]
\end{align}

We will use the second order approximation as an  example $(n=2)$. In this case, the approximation would proceed as follows,
\begin{enumerate}
\item Set the Laplace transform of the approximated kernel to,
$$\Theta_2(s)=(s^2 - s B_0 - B_1)^{-1}(s C_0 + C_1).$$
\item Solve for the coefficients using the moments:
\begin{align}
\left[
\begin{array}{cc}
-M_{\infty} & M_0\\
M_0 & M_1
\end{array}
\right]
\left[ \begin{array}{c}B_1\\B_0\end{array} \right]=
\left[\begin{array}{c} M_1\\M_2\end{array}\right],
\quad C_0=M_0,\quad C_1 = -B_1 M_{\infty}.
\end{align}
\item The approximate kernel function in the real time domain can be expressed as:
\begin{align}\label{eq: BB}
\theta_2(t) \approx [0 \quad I]e^{ Bt}
 C,
\quad {\rm where}\quad
 B = 
\left[\begin{array}{cc}
0 & B_1\\
I & B_0
\end{array}\right], \quad  C=\left[ \begin{array}{c} C_1\\C_0\end{array}\right]
.
\end{align}

\end{enumerate}

\noindent{Remark1 :} Once $B$ and $C$ are computed, the variance of the random noise $\zeta$ in the stochastic equation $\dot{z}= B z + C p + \zeta(t)$, as well as the variance of the $z(0)$, will be chosen based on these two matrices to satisfy the FDT. Such computation is quite involved in general.  Fortunately, as we will show, the subspace projection approach simplifies this effort considerably. 

{\blue 
\noindent{Remark 2:} Although one can increase the order of the approximation by simply introducing more moments, there remains an important practical problem, that is,  the condition number of the matrix in equation (\ref{multiB}) increases rapidly as the order increases. We hereby list the condition numbers in the following Table \ref{tb: cond} for a test problem.

\begin{table}[h]
\caption{Condition numbers of the matrix in (\ref{multiB}) in the moment matching procedure }\label{tb: cond}
\begin{center}
\begin{tabular}{|c|c|c|c|c|c|c|}
\hline
Approximation order & 2 & 3 & 4 & 5 & 6 & 7\\
\hline
Matrix condition number & 4.98E03 & 1.59E12& 4.57E14 &1.11E22&5.58E27&1.76E33\\
\hline

\end{tabular}
\end{center}
\end{table}%

}

{\brown We now turn to the Krylov subspace projections. }

\subsection{First Order Subspace Projection $n=1$} 

As the first approximation, 
we choose the subspaces  
\begin{equation}\label{eq: kr1}
  V=R, \text{and}\;  W=D^{-\intercal}L^\intercal.
\end{equation}

We  show that the resulting approximate kernel function is the same as that from the moment matching approach.


\begin{theorem}\label{thm: 1stEQ}
{\red By taking  $V=R$ and  $W=D^{-\intercal}L^\intercal$ in the Galerkin projection, 
the memory kernel $\wh \theta_1(t)$ in the projected dynamics \eqref{eq: GLE'}} 
is equivalent to that from the first order moment matching method.
 In particular,  two moments are matched exactly by the approximate kernel functions,
\begin{equation}
 \wh\theta(0) = M_0, \quad {\text and}\int_0^{+\infty} \wh{\theta}_1(t) dt =M_\infty.
\end{equation}

\end{theorem}
\begin{proof}

With direct computation, we get from \eqref{eq: th-approx} that,
$$
\widehat \theta_1 (0) = LV\widehat {M}^{-1} W^{\intercal} R, 
\quad
\int_0^{+\infty} \wh{\theta}_1(t) dt= -LV\widehat{D}^{-1}W^\intercal R.
$$
By the particular choice of $V$ and $W$ \eqref{eq: kr1}, we have $\widehat{M}=W^\intercal R$. Therefore,
$$
\wh \theta_1(0)=LR=M_0,\quad \int_0^{+\infty} \wh{\theta}_1(t) dt=W^\intercal R =LD^{-1}R=M_\infty.
$$
\end{proof}

\begin{theorem}\label{thm: 1st}
{\red By taking  $V=R$ and  $W=D^{-\intercal}L^\intercal$ in the Galerkin projection, }
the projected dynamics  \eqref{eq: GLE'} will automatically satisfy the second FDT \eqref{eq: fdt3}.
\end{theorem}
\begin{proof}
 We need to show that \ref{eq: th2} is satisfied by this choice of $W$ and $V$ in this case.
Given  (\ref{eq: sig}) and \ref{eq: cond-1}, we have
$$
-W^\intercal \Sigma W=-\widehat M\widehat \Sigma \widehat{M}^{-\intercal}= k_BT (\widehat D \widehat Q \widehat M^\intercal+ \widehat M \widehat Q\widehat D^\intercal).
$$
Notice that since $W=D^{-\intercal}L^\intercal$,  one has $\widehat{D}=LV$. In addition, from Equation (\ref{eq: Lya}),  we have
$$
LQW + W^\intercal Q L^\intercal = LV\widehat{Q}\widehat{M}^\intercal + \widehat{M}\widehat{Q}V^\intercal L^\intercal.
$$
It is clear that on both sides, it is a summation of a matrix and its transpose. By moving terms we find,
\begin{equation}\label{eq: tmp}
LQW -LV\widehat{Q}\widehat{M}^\intercal =   \widehat{M}\widehat{Q}V^\intercal L^\intercal - W^\intercal Q L^\intercal
\end{equation}
and  \ref{eq: th2} would hold if either  side equals to zero. {\brown We will examine the two terms on the right hand side.  }

Since $\Gamma\equiv\gamma I $, we have $\Gamma_{12}=0$. Further 
notice that,
\begin{align*}
D^{-1}=\left[
\begin{array}{cc}
-A_{22}^{-1}\Gamma_{22} & -A_{22}^{-1}\\
I&0
\end{array}
\right].
\end{align*}
By direct calculations, 
the second term on the right hand side can be simplified to,
$$ W^\intercal Q L^\intercal  = LD^{-1}QL^\intercal=-A_{12}A_{22}^{-1}\Gamma_{22}A_{22}^{-1}A_{21}=-M_{\infty}.$$ 

Regarding the first term on the right hand side of \eqref{eq: tmp}, it can be directly verified that,
$$
\wh M = LD^{-1}R = - M_{\infty}, \quad \wh D = LR=M_0,
$$
which are both symmetric matrices. Further, by using Equation (\ref{eq: sig}), we get that,
$$
\wh \Sigma = 2k_B TM_{\infty}^{-1}.
$$
Then the Lyapunov Equation (\ref{eq: cond-1}) becomes
$$
k_B T (M_{\infty}^{-1} M_0 \wh Q + \wh Q M_0^\intercal  M_{\infty}^{-1})=2k_BTM_{\infty}^{-1},
$$from which we obtain the solution $\wh Q=M_0^{-1}$. 

Therefore the first term on the right hand side of \eqref{eq: tmp} becomes\
{
$$
\wh M \wh Q V^\intercal L^\intercal=-M_{\infty} M_0^{-1} R^\intercal L^\intercal = -M_{\infty},
$$
}
which would cancel the second term and complete the proof.
\end{proof}



\subsection{Second Order Subspace Projection $n=2$}


{\brown We now extend the subspace by choosing,}
\begin{equation}\label{eq: kr2}
V=[R, DR], \;\text{and}\; W=[D^{-\intercal} L^\intercal, L^\intercal].
\end{equation}
 
 As a result, 
the two matrices $\wh{M}$ and $\wh{D}$ in the Galerkin formulation are given by,
\begin{equation}
 \wh{M} =\left[\begin{array}{cc} - M_\infty & M_0 \\M_0 & M_1 \end{array}\right], \quad
 \wh{D} =\left[\begin{array}{cc}M_0 & M_1 \\M_1 & M_2 \end{array}\right].
\end{equation}
It's easy to check that $\wh{M}^{-1} \wh{D}= B$, as in equation (\ref{eq: BB}). Within this extended approximation, the approximate memory function is given by,
\begin{equation}\label{eq: thh2}
\wh{\theta}_2(t)= [M_0 \quad M_1] e^{ Bt}\wh{M}^{-1} 
\left[\begin{array}{c} -M_\infty \\M_0 \end{array}\right]
=[M_0 \quad M_1] e^{ Bt}
\left[\begin{array}{c} I \\0 \end{array}\right].
\end{equation}

\smallskip

We first show that this approximation is equivalent to the moment matching procedure.  It is straightforward to verify that the approximate kernel, denoted by $\theta_2$, from the moment matching procedure, should satisfy the following second order differential equation:
\begin{align}\label{eq: 2ndODE}
\ddot{\theta}_2(t) = B_0\dot{\theta}_2(t) + B_1\theta_2(t), \quad\theta_2(0)=M_0, \quad\dot{\theta}_2(0)=M_1.
\end{align}

We now show that the kernel function $\wh\theta_2(t)$ follows the same equation. Thanks to the uniqueness, we can then conclude the equivalence.  The key observation is that,
$$
[M_0 \quad M_1] B = [M_1 \quad M_2].
$$
As a result, it can be quickly verified that
$$
B_1[M_0\quad M_1]+B_0[M_0\quad M_1] B=B_1[M_0 \quad M_1]+B_0[M_1\quad M_2] = [M_1\quad M_2] B = [M_0\quad M_1] B^2,
$$
which combined with \eqref{eq: thh2} would lead to 
\begin{align}
\ddot{\wh \theta}_2(t) = B_0\dot{\wh \theta}_2(t) + B_1\wh \theta_2(t), \quad\wh \theta_2(0)=M_0, \quad\dot{\wh \theta}_2(0)=M_1.
\end{align}
Therefore, we have this following theorem.
\begin{theorem}\label{thm: 2ndEQ}
{\red The reduced model    \eqref{eq: GLE'}} from the Galerkin projection with the choice of $V=[R, DR]$, and $W=[D^{-\intercal}L^\intercal, L^\intercal]$ produces an approximate memory kernel function, which  is equivalent to  that from the second order moment matching procedure.
\end{theorem}
Furthermore, we have,
\begin{theorem}\label{thm: 2ndFDT}
{\red The projected system \eqref{eq: GLE'}} with the choice of $V=[R, DR]$, and $W=[D^{-\intercal}L^\intercal, L^\intercal]$ will automatically satisfy the second FDT \eqref{eq: fdt3}.
\end{theorem}

\begin{proof}
We only need to justify \ref{eq: th2}. It is straightforward to show that,
 $$
 W^\intercal Q L^\intercal = 
  \left[
  \begin{array}{c}
  LD^{-1}\\L
  \end{array}
  \right]
  \left[
  \begin{array}{cc}
  A_{22}^{-1}&0\\0&I
  \end{array}
  \right]
  \left[
  \begin{array}{c}
  A_{21}\\0
  \end{array}
  \right] = W^\intercal R =  \left[
  \begin{array}{c}
  -M_{\infty}\\M_0
  \end{array}
  \right].
 $$
 With the choice of $V$,  we have
 $$V^\intercal L^\intercal =\left[\begin{array}{c}M_0\\M_1^\intercal\end{array}\right]
 =\left[\begin{array}{c}M_0\\0\end{array}\right].$$
 Notice $M_1=0$, which can be verified  by direct calculation.
Then by some direct calculations with  the representation of the covariance matrix, we have
$$
\wh M\wh \Sigma \wh M^\intercal = W^\intercal \Sigma W = \left[
\begin{array}{cc}
2k_BT M_{\infty} & 0\\0 &0
\end{array}
\right]
=- k_B T (\wh D \wh Q \wh M^\intercal + \wh M \wh Q \wh D^\intercal).
$$
Meanwhile, we have, 
$$
\wh M  =\left[ \begin{array}{cc} -M_{\infty}&  M_0\\M_0&0\end{array}\right],
$$
which gives,
$$
\wh M^{-1}  =\left[ \begin{array}{cc} 0&  M_0^{-1}\\M_0^{-1}&M_0^{-1}M_{\infty}M_0^{-1}\end{array}\right],\quad
{
\wh \Sigma=2k_BT\left[\begin{array}{cc} 0 & 0\\ 0& M_0^{-1}M_{\infty}M_0^{-1} \end{array} \right].
}
$$
Now we  solve the Lyapunov equation and we find that,
$$
 \wh Q = \left[\begin{array}{cc} M_0^{-1} & 0\\0&- M_2^{-1}\end{array}\right].
$$
With $\wh Q$ available,  it can be verified that 
$$
\wh M \wh Q V^\intercal L^\intercal = \left[\begin{array}{c}M_0\\0\end{array}\right] = W^\intercal Q L^\intercal,
$$
which is our condition {\bf B}, thus it completes the proof. 

\end{proof}

\subsection{Generalization to High Order Approximation ($n \ge 2$)}

Inspired by the previous choices, we consider
\begin{align}\label{eq: choice}
V=[R, DR, \dots, D^{n-1}R],\quad W=[D^{-\intercal}L^\intercal, L^\intercal,D^{\intercal} L^\intercal,\dots,(D^\intercal)^{n-2}L^\intercal],
\end{align} 
and apply Galerkin projection to the two subspaces generated by the columns of these two matrices.

 The  corresponding matrices $\wh{M}$, $\wh{D}$, $ B$ and $W^\intercal R$ are given by, respectively,
\begin{align}
\wh{M}=&\left[
\begin{array}{cccc}
-M_{\infty} &M_0&\dots&M_{n-2}\\
M_0&M_1&\dots &M_{n-1}\\
\vdots\\
M_{n-2}&\dots&&M_{2n-3}
\end{array}
\right],
\quad
\wh D=\left[
\begin{array}{cccc}
M_0&M_1&\dots&M_{n-1}\\
M_1&M_2&\dots&M_{n}\\
\vdots\\
M_{n-1}&\dots&&M_{2n-2}
\end{array}
\right],\\
 B =\wh M^{-1} \wh D=&
\left[
\begin{array}{ccccc}
0&0&\dots&0&B_{n-1}\\
I&0&\dots&0&B_{n-2}\\
0&I&\dots&0&\\
\vdots\\
0&0&\dots&I&B_0
\end{array}
\right],\quad
W^\intercal R = \left[\begin{array}{c}
-M_{\infty}\\M_0 \\ \vdots \\ M_{n-2}
\end{array}
\right].
\end{align}
Therefore, the approximate kernel under the Galerkin projection can be expressed as,
\begin{equation}\label{eq: whthetan}
\wh{\theta}_n(t)= [M_0 \quad M_1\quad \dots\quad M_{n-1} ] e^{ Bt}
\left[\begin{array}{c} I\\ 0\\\vdots\\0 \end{array}\right].
\end{equation}

Meanwhile, the high order approximate memory kernel from the moment matching procedure satisfies the $n^{th}$ order differential equation:
$$
\theta^{(n)}_n(t)=B_0 \theta^{(n-1)}_n(t)+B_1\theta^{(n-2)}_n(t)+\dots B_{n-1}\theta_n(t), \quad\theta_n(0)=M_0, \quad\dots,\quad\theta^{(n-1)}_n(0)=M_{n-1}.
$$

We first show that these approximate kernel functions are the same.

{\theorem  The function $\wh \theta_n(t)$ {\red in equation (\ref{eq: whthetan})} is equivalent to {\red the function $\theta_n(t)$ generated from moment matching procedure as described in section 4.1}. In particular, it also satisfies the initial-value problem,
$$
\wh\theta^{(n)}_n(t)=B_0 \wh\theta^{(n-1)}_n(t)+B_1\wh\theta^{(n-2)}_n(t)+\dots B_{n-1}\wh\theta_n(t), \quad\wh\theta_n(0)=M_0, \quad\dots,\quad\wh\theta^{(n-1)}_n(0)=M_{n-1}.
$$
\begin{proof}
Each $M_i$ is a $m$ by $m$ matrix, and the dimension of   $\wh \theta_n(t)$ is also $m\times m$.
For simpler notations, we will denote $[M_i\quad M_{i+1}\quad \dots \quad M_{i+n-1}]=G_i$.
If  we can show that
$$
G_0 B^n = B_0G_0 B^{n-1} + B_1G_0 B^{n-2} + \dots + B_{n-1} G_0,
$$
this will prove $\wh \theta_n(t)$ satisfies the same differential equation.
Notice that the recursive relation $$G_i  B=G_{i+1},\quad {\rm for}\quad i=0,\dots n-2,$$ comes straightforward since $\wh M  B = \wh D$.
Then it remains to check that
$$
G_{n-1} B = B_0 G_{n-1}+B_1 G_{n-2} + \dots +B_{n-1}G_0.
$$
We will take a closer look at each block elements. The first block on the left hand side is $M_n$, and on the right hand side, we have $\sum_{i=0}^{n-1}M_iB_{n-1-i}$. They are equal  due to the equation $\wh M  B = \wh D$. In fact, all other blocks except the last one can be shown from the same equation.
 The last block automatically equal to each other since they have exactly the same representation. 
 
 For the initial conditions, they can be easily verified using  $\wh M  B = \wh D$.
\end{proof}
}
What we will study next is whether this formulation also obeys the second FDT. However, we are not able to prove the general case due to the lengthy calculations involved. We are able to prove the consistency for  $n \le 5$. The following few results are useful for the verification. Numerical tests suggest that the consistency holds also for higher order cases.

\begin{lemma} \label{lem: symmetric}
The moments of the memory function are all symmetric matrices. As a result, $\wh M$, and $\wh D$ {\red as defined in equation (\ref{eq: MD})} are also symmetric matrices. 
\end{lemma}
\begin{proof}
 We only need to show all moments $M_i$ are symmetric. 
Recall that 
$$
D=\left[\begin{array}{cc} 0 &I \\ -A_{22} & -\Gamma_{22}\end{array} = 
\right]
=\left[\begin{array}{cc}
0&-I\\-I&\Gamma_{22}
\end{array}
\right]\left[\begin{array}{cc}
A_{22}&0\\0&-I
\end{array}
\right],
$$
and 
$$
R=\left[\begin{array}{c}
A_{22}^{-1} A_{21}\\-\Gamma_{21}
\end{array}
\right]=\left[\begin{array}{cc}
A_{22}^{-1}&0\\0&-I
\end{array}
\right]\left[\begin{array}{c}
 A_{21}\\\Gamma_{21}
\end{array}
\right].
$$
Therefore for $i>0$,
\begin{align*}
M_i=LD^iR =& 
\begin{array}{cc}
[A_{12} & \Gamma_{12}]
\end{array}
\left[\begin{array}{cc}
0&-I\\-I&\Gamma_{22}
\end{array}
\right]
\left[\begin{array}{cc}
A_{22}&0\\0&-I
\end{array}
\right]\cdots
\left[\begin{array}{cc}
0&-I\\
-I&\Gamma_{22}
\end{array}
\right]
\left[\begin{array}{cc}
A_{22}&0\\0&-I
\end{array}
\right]
\left[\begin{array}{cc}
A_{22}^{-1}&0\\0&-I
\end{array}
\right]
\left[\begin{array}{c}
 A_{21}\\\Gamma_{21}
\end{array}
\right] \\
= &
\begin{array}{cc}
 [A_{12} & \Gamma_{12}]
\end{array}
\left[\begin{array}{cc}
0&-I\\-I&\Gamma_{22}
\end{array}
\right]
\left[\begin{array}{cc}
A_{22}&0\\0&-I
\end{array}
\right]\cdots
\left[\begin{array}{cc}
0&-I\\
-I&\Gamma_{22}
\end{array}
\right]
\left[
\begin{array}{c}
 A_{21}\\ 
 \Gamma_{21}
\end{array}
\right],
\end{align*}
which is clearly symmetric.
At the same time, it is straightforward to see that $M_0$ is symmetric by direct calculation. Finally, 
\begin{align*}
M_{\infty}&=\begin{array}{cc}[A_{12} & \Gamma_{12}]\end{array}
\left[
\begin{array}{cc}
-A_{22}^{-1}\Gamma_{22} & -A_{22}^{-1}\\
I&0
\end{array}
\right]\left[\begin{array}{cc}
A_{22}^{-1}&0\\0&-I
\end{array}
\right]\left[\begin{array}{c}
 A_{21}\\\Gamma_{21}
\end{array}
\right]\\
&=\begin{array}{cc}[A_{12} & \Gamma_{12}]\end{array}
\left[
\begin{array}{cc}
-A_{22}^{-1}\Gamma_{22}A_{22}^{-1} & A_{22}^{-1}\\
A_{22}^{-1}&0
\end{array}
\right]
\left[\begin{array}{c}
 A_{21}\\\Gamma_{21}
\end{array}
\right],
\end{align*}
 is symmetric as well.
\end{proof}

\begin{lemma}\label{tm: 7}
Assume that $\wh M$ is invertible. 
Condition {\bf B} is equivalent to, 
\begin{align}\label{eq: lemma}
\wh Q \left[\begin{array}{c}
M_0\\M_1\\\vdots\\M_{n-1}
\end{array}
\right]=\left[\begin{array}{c}
I\\0\\\vdots\\0
\end{array}
\right].
\end{align}
\end{lemma}
\begin{proof}

When $\Gamma=\gamma I$, $\Gamma_{12}=0$, we have the identity $QL^\intercal =R$. It is also easy to see that due to symmetry from Lemma 1, one has,  
$$
V^\intercal L^\intercal = \left[\begin{array}{c}
M_0^\intercal \\M_1^\intercal \\\dots\\M_{n-1}^\intercal
\end{array}
\right]= \left[\begin{array}{c}
M_0\\M_1\\\dots\\M_{n-1}
\end{array}
\right].
$$
Therefore, \ref{eq: th2} becomes
$$
\wh M\wh Q \left[\begin{array}{c}
M_0\\M_1\\\dots\\M_{n-1}
\end{array}
\right]= W^\intercal R=
\left[\begin{array}{c}
-M_\infty\\M_0\\\dots\\M_{n-2}
\end{array}
\right].
$$
Multiplying both sides by $\wh M^{-1}$ (with the assumption that $\wh M$ is invertible), we arrive at equation (\ref{eq: lemma}). 
\end{proof}

\begin{lemma}
Let $\wt \Sigma = W^\intercal \Sigma W$, which has dimension $nm \times nm$. If it is partitioned into a block matrix with each block having dimension $m\times m$ , then the  block elements have the following recurrence relations:
\begin{align}
&\wt{\Sigma}_{1,1} = 2M_{\infty},\quad \wt{\Sigma}_{i,2}=0, \quad \wt{\Sigma}_{2,i}=0,\\
&\wt{\Sigma}_{ij} = -\frac{1}{\gamma} \wt\Sigma_{i+1,j}- \frac{1}{\gamma} \wt\Sigma_{i,j+1}-2k_BT M_{i+j-3}, \quad i,j >2.
\end{align}
\end{lemma}

As a result, the elements of $\wt \Sigma$ can be constructed column by column using the recurrence relation. They can be expressed in terms of the moments $M_i$s. The next lemma shows that the moments also exhibit a recurrence relation, which can be exploited to make the calculation a bit easier.

\begin{lemma}
The moments $M_i=LD^iR$ can be written as a linear combination of matrices $A_{12}A_{22}^kA_{21}$,
$$
M_i = \sum_{k=0}^{\lfloor \frac{i}{2}\rfloor-1} c_{i,k} A_{12} A_{22}^k A_{21}.
$$
\end{lemma}

The proof of these lemmas can be found in the appendix.

{
\theorem {\red The reduced system  \eqref{eq: GLE'}} from the  Petrov-Galerkin projection obeys the second FDT for orders  { $n\leq 5$}.
}
\begin{proof}
It now becomes clear that in order to check wether  the second FDT holds for high order approximation system, one only needs to show equation (\ref{eq: lemma}). On the other hand, we know $\wh Q$ is the solution to Lyapunov equation (\ref{eq: cond-1}), which uniquely determined. Therefore under the assumption that $\wh D$ is nonsingular, $\wt Q=\wh Q \wh D$ is also uniquely determined. This also leads to the following equation based on the fact that $\wh M$ is symmetric.
$$
k_B T (\wt Q\wh M + \wh M \wt Q^\intercal) = -\wh M \wh \Sigma \wh M = -W^\intercal \Sigma W=-\wt {\Sigma}.
$$
Now the goal is to compute the exact form of $\wt Q$. We will present the expression of $\wt Q$ and $W^\intercal \Sigma W$ for $n=3, 4, 5$, and readers can substitute those forms into the equation above to verify. There are some identities needed in order to complete the verification, which we will present in the Appendix.

For $n=3$,
\begin{align*}
\wt \Sigma= k_B T\left[
\begin{array}{ccc}
2M_{\infty} &0 & -2\gamma M_0\\
0&0&0\\
-2\gamma M_0&0&-2\gamma M_2
\end{array}
\right],
\quad
\wt Q=\left[\begin{array}{ccc}
I & 0 &0\\
0&-I &0\\
0&2\gamma I&I
\end{array}
\right].
\end{align*}

For $n=4$,
\begin{align*}
\wt \Sigma = k_B T\left[
\begin{array}{cccc}
2M_{\infty} &0 & -2\gamma M_0& 2\gamma^2 M_0\\
0&0&0&0\\
-2\gamma M_0&0&-2\gamma M_2& -2\gamma M_3\\
2\gamma^2M_0&0 & -2\gamma M_3&2\gamma^2M_3\\
\end{array}
\right],
\quad
\wt Q=\left[\begin{array}{cccc}
I & 0 &0 &0\\
0&-I &0&0\\
0&2\gamma I&I&0\\
0 & -2\gamma^2 I &-2\gamma I&-I 
\end{array}
\right].
\end{align*}

For $n=5$,
\begin{align*}
&\wt \Sigma = k_B T\left[
\begin{array}{ccccc}
2M_{\infty} &0 & -2\gamma M_0& 2\gamma^2 M_0&-2\gamma^3 M_0 - 2\gamma M_2\\
0&0&0&0&0\\
-2\gamma M_0&0&-2\gamma M_2& -2\gamma M_3&2\gamma^2M_3\\
2\gamma^2M_0&0 & -2\gamma M_3&2\gamma^2M_3 &2 \gamma^2 M_4\\
-2\gamma^3 M_0 - 2 \gamma M_2 &0&-2\gamma M_4&2\gamma^2M_4 & -2\gamma^3 M_4 - 2\gamma^2 M_5 - 2\gamma M_6
\end{array}
\right],
\\
&\wt Q=\left[\begin{array}{ccccc}
I & 0 &0 &0&0\\
0&-I &0&0&0\\
0&2\gamma I&I&0&0\\
0 & -2\gamma^2 I &-2\gamma I&-I &0\\
0 & 2\gamma^3 I& 4\gamma^2 I& 4\gamma I&I
\end{array}
\right].
\end{align*}

\end{proof}

\section{Numerical Implementation}

In this section, we will describe  the numerical implementation of the Krylov subspace projection method.
In the previous section, we have studied properties of the projected dynamics with particular choices of   $V$ and $W$. 
However,  as is well known \cite{bai2002krylov}, a direct implementation using those matrices often leads to ill-conditioned matrices. This has clearly been shown in Table 1. 
A much more robust approach is to obtain orthogonal basis by using appropriate orthogonalization algorithms.

Let us first introduce the notations for these two Krylov subspaces for an $n$th order approximation. 
$$
\mathcal K_n(D, R) = span\{R, DR, \dots, D^{n-1} R\}, \;\; \mathcal K_n(D^\intercal,D^{-\intercal} L^\intercal)=span\{D^{-\intercal}L^\intercal, L^\intercal, \dots, (D^{n-2})^\intercal L^\intercal\}.
$$

\subsection{Block Lanczos Algorithms (BLBIO)}
 We will adopt  the non-symmetric block Lanczos algorithms from  \cite{loher2006reliable}
 to generate orthogonal basis $V=[ V_1, \dots, V_n]$ and $W=[ W_1, \dots, W_n]$ for $\mathcal K_n(D, R)$ and $\mathcal K_n(D^\intercal,D^{-\intercal} L^\intercal)$, respectively. 

The Lanczos algorithm proceeds as follows. 
Choose $V_1=R, W_1= D^{-\intercal} L^\intercal$, and let $\delta_1 = W_1^\intercal V_1$, and for $k=1, 2, \dots$ compute 
\begin{align}
&\delta_k^A = W_k ^\intercal D V_k,\\
&\alpha_k = \delta_k^{-1}\delta_k^A,\quad \wt {\alpha}_k=\delta_k^{-\intercal}(\delta_k^A)^{\intercal},\\
 &\beta_{k-1}=\delta_{k-1}^{-1}\wt \gamma_{k-1}^\intercal\delta_k,\quad \wt \beta_{k-1}=\delta_{k-1}^{-\intercal}\gamma_{k-1}^\intercal\delta_k^\intercal,\quad ({\rm if} \; n>0)\\
 &V_{tmp}=DV_k - V_k\alpha_k-V_{k-1}\beta_{k-1},\quad 
 W_{tmp}=D^\intercal W_k-W_k \wt  \alpha_k - W_{k-1}\wt  \beta_{k-1},\\
&\delta_{tmp}=W_{tmp}^\intercal V_{tmp}\\
&{\rm choose}\; \gamma_k, \wt  \gamma_k\; {\rm and} \;\delta_{k+1}, \quad{\rm s.t.}\quad 
\wt {\gamma}_k^\intercal \delta_{k+1} \gamma_k = \delta_{tmp}
\end{align}

Several possible choices have been recommended in  \cite{loher2006reliable} for $ \gamma_k, \wt  \gamma_k${and} $\delta_{k+1}$.  We found that the QR factorization with column pivoting for{ $V_{tmp}$ and $W_{tmp}$} is quite robust. Namely,  
$$
V_{tmp} P=UR,\quad W_{tmp} \wt P = \wt U \wt R.
$$
Then we choose 
$$
V_{k+1}=U, \quad W_{k+1}=\wt U,\quad \gamma_{k}=RP^\intercal,\quad \wt \gamma_{k}=\wt R \wt P^\intercal.
$$

By following this algorithm, we obtain the orthogonality properties among the basis vectors of the Krylov subspaces. In particular, the matrix $\wh M$ is diagonal, and the matrix $\wh D$ is block-tridiagonal. 
As a result, the SDEs for the auxiliary variable $ z$ (\ref{eq: projD}) involves sparse matrices.
\subsection{Implementation without $\Psi$}
The implementation of the algorithm requires the matrices $L$, $V$, $W$, and $R$,  all involving the $\Psi$ matrix as part of the construction. Constructing $\Psi$  is usually not feasible for large systems. Here we present an algorithm that does not involve   $\Psi$.

\smallskip

Let's first derive a few useful identities involving the $\Psi$ matrix.
We start with,
$$
\left[\begin{array}{c} \Phi^\intercal \\\Psi ^\intercal \end{array}\right] A [\Phi \quad\Psi]=
\left[\begin{array}{cc} A_{11}&A_{12}\\A_{21}&A_{22}\end{array}\right],
\quad\quad
\left(\left[\begin{array}{c} \Phi^\intercal \\\Psi ^\intercal \end{array}\right] A [\Phi \quad\Psi]\right)^{-1}
=\left[\begin{array}{c} \Phi^\intercal \\\Psi ^\intercal \end{array}\right] A^{-1} [\Phi\quad \Psi].
$$

Using a block inversion formula, we get
$$
(\Phi^\intercal A^{-1} \Phi)^{-1} = A_{11} - A_{12}A_{22}^{-1}A_{21}.
$$
By left multiplying the equation by $A^{-1}\Phi$, together with the identity $\Phi \Phi^\intercal = I - \Psi \Psi^\intercal$,  we find that,
\begin{align}\label{eq: a22a21}
A^{-1}\Phi(\Phi^\intercal A^{-1}\Phi)^{-1} = \Phi - \Psi A_{22}^{-1} A_{21}.
\end{align}
Next, right multiplying the above equation by $\Phi^\intercal A^{-1}$, we arrive at,
\begin{align}\label{eq: pa22p}
A^{-1}\Phi(\Phi^\intercal A^{-1}\Phi)^{-1} \Phi^\intercal A^{-1}= \Phi \Phi^\intercal A^{-1} - \Psi A_{22}^{-1} \Psi^\intercal A \Phi \Phi^{\intercal} A^{-1} = \Phi \Phi^\intercal A^{-1} - \Psi A_{22}^{-1} \Psi^\intercal+ (I-\Phi\Phi^{\intercal})A^{-1}.
\end{align}

Now we define, $$\wt D = \left[\begin{array}{cc}\Psi&0\\0&\Psi\end{array}\right] D \left[\begin{array}{cc}\Psi^\intercal&0\\0&\Psi^\intercal\end{array}\right], \quad\quad \wt R = \left[\begin{array}{cc}\Psi&0\\0&\Psi\end{array}\right] R, \quad \wt L =\Phi^\intercal \big[ A, \;\Gamma\big] \left[\begin{array}{cc}\Psi&0\\0&\Psi\end{array}\right]  D^{-1} \left[\begin{array}{cc}\Psi^\intercal &0\\0&\Psi^\intercal\end{array}\right]$$

We start with the following observation,
\begin{lemma} 
The following relation holds between the two Krylov subspaces, $ \mathcal{K}_n(D,R)$ and $\mathcal K_n(\wt D, \wt R)$:
$$\left[\begin{array}{cc}\Psi&0\\0&\Psi\end{array}\right] 
\mathcal{K}_n(D,R) = 
\mathcal K_n(\wt D, \wt R).$$
Similarly,
$$\left[\begin{array}{cc}\Psi&0\\0&\Psi\end{array}\right] \mathcal K_n (D^\intercal , D^{-\intercal}L^\intercal)
=\mathcal K_n( \wt D^{\intercal},\wt L^\intercal). $$
\end{lemma}

{\brown With these observations, we show that:} 

\begin{theorem}
The Lanczos algorithm, the Galerkin projection, and  the sampling of the noise, can be done without $\Psi$.
\end{theorem}
\begin{proof}

First it can be directly shown that, 
 $$
 \wt D = \left[ \begin{array}{cc} 0 & \Psi \Psi^\intercal\\
 -\Psi \Psi^\intercal A \Psi \Psi^\intercal & -\Psi \Psi^{\intercal} \Gamma \Psi\Psi^{\intercal}\end{array}\right].
 $$
Thanks again to the identity 
\begin{equation}\label{eq: pp}
 \Psi \Psi^\intercal =  I - \Phi \Phi^\intercal,
\end{equation}
 we  can evaluate $\Psi \Psi^\intercal$ through the matrix $\Phi$. Therefore the calculation of $\wt D$ can be done without $\Psi$.
 
Secondly, to compute $\wt L$, we notice that the terms involving $\Psi$ are $\Psi A_{22}^{-1}\Gamma_{22}\Psi^{\intercal}$, $\Psi\Psi^\intercal$, and $\Psi A_{22}^{-1}\Psi^\intercal$, and these terms can be represented without $\Psi$ from Equation (\ref{eq: pa22p}) and (\ref{eq: pp}).  The calculation of $\wt R$ is similar.
 
 Thirdly, we see that the solution of the projected dynamics \eqref{eq: projD} enters the coarse-grained dynamics \eqref{eq: Krylov} via a matrix multiplication by $LV.$
It is straightforward to write $L$ as $$L=\Phi^\intercal \big[ A, \;\Gamma\big]\left[\begin{array}{cc}\Psi&0\\0&\Psi\end{array}\right],$$ which means that for the term $LV$, we can actually compute  $\Phi^\intercal  \big[ A, \;\Gamma\big]\wt V,$ where $\wt V$ is constructed using  using the block Lanczos from space  $\mathcal K_n(\wt D, \wt R)$.
 
 Now let $\wt V$ and 
 $\wt W$ be the basis generated from the orthogonalization of the new Krylov subspaces $\mathcal K_n(\wt D, \wt R)$ and $\mathcal K_n(\wt D^\intercal, \wt L^\intercal )$, respectively.
 Therefore, the matrices  $\wh M = W^\intercal V = \wt W^\intercal \wt V$, $\wh D = W^\intercal D V=\wt W^\intercal \wt D \wt V$  and $W^{\intercal}R=\wt W^{\intercal} \wt R$ can all be generated without introducing $\Psi$.

 Finally, it remains to show that the sampling of the noise does not have to involve $\Psi$, which is clearly true since the noise is represented as $\wh M^{-1} W^\intercal \left[\begin{array}{c}0\\\Psi^\intercal  f\end{array}\right]$, and $W^{\intercal} \Psi^\intercal = \wt W^\intercal$.

\end{proof}
 It is a trivial, but important point in practice that 
 in the numerical implementation, it is not necessary to store the full matrix $\Phi\Phi^\intercal$. For a given vector $u$, the multiplication $\Phi \Phi^\intercal u$ can be done through $\Phi (\Phi^\intercal u)$.

{\blue
\subsection{A Summary of the Galerkin Projection}
The Galerkin projection method can be summarized as follows,
\begin{enumerate}
\item Choose appropriate basis matrix $\Phi$. 
\item  Pick the order of approximation $n \ge 1$. Use the block-Lanczos algorithm to determine the orthogonal basis $V$ and $W$,
 for the Krylov subspaces $\mathcal K_n(D, R)$ and $\mathcal K_n(D^\intercal,D^{-\intercal} L^\intercal)$, respectively.
\item Solve the stochastic differential equations \eqref{eq: Krylov'2}, where $\wh M, \wh D, \bs {\wh f}$ are defined from equations (\ref{eq: MD}) and (\ref{eq: whf}).
The initial variance of $z(t)$ is determine from \ref{eq: cond-1}.
\end{enumerate}
Clearly, this procedure avoided manual constructions of the reduced model. This choice of the Krylov subspaces guarantees that the FDT is satisfied (through \ref{eq: th2}), at least till the fifth order of approximation (numerical tests indicate that this is true for higher order cases). 

}
\section{Numerical Test}
We test our algorithm on the example considered in  \cite{ma2016derivation}. We simulate the dynamics of the protein Chignolin {\red (PDB id 1uao)} at temperature $T=298$ for .4 nano seconds. 
 {\red The system is set up in solvation, modeled by the generalized Born (GB) model and simulations have been conducted  in TINKER \cite{ponder2004tinker} using force field CHARMM22.}
For the surrounding bath, we considered the case $\gamma=91 ps^{-1}$ which corresponds to water solvant \cite{ponder2004tinker} and a  low friction  case $\gamma=5 ps^{-1}$. 
In the latter case,  the kernel function exhibits nontrivial behavior \cite{ma2016derivation}: it tends to be more oscillatory  compared to the former case. 
{By calculating the eigenvalues of $A$ we have identified the under-damped regime to be $\gamma < 13.4$ and the over-damped regime to $\gamma > 997.7$.  } 
 { Data are collected to compute the PCA matrix $A=k_BT\langle x, x^\intercal\rangle^{-1}$. The projection matrix are composed of RTB basis \cite{TaGaMaSa00}, since there are 10 residues in Chignolin, the dimension of the coarse-grained variables is 60. The explicit forms of the basis functions in $\Phi$ each each translational and rotational mode can be found in \cite{LiXian2014}. }
 
{
 We first  present the numerical result for $\gamma = 91$ in Figure \ref{Fig: g91}. 
 On the left panel, we showed the comparison of approximating memory function, from order two to order seven. The right panel of the figure provides the comparison of time correlation of the momentum. Both exact plots are obtained  by running the full model. The order of approximations, $n,$  is defined as the order of Krylov subspaces, which is equivalent to the order of the rational functions in the moment matching approach. Since the kernel function $\theta(t)$ is matrix-valued, we chose the sixth diagonal, $\theta_{6,6}(t)$ for the comparison, this index corresponds to the third rotational component of the first residue. We can observe the approximation is satisfactory for $n\le 5$. 
\begin{figure}[htp]
\centering
\includegraphics[scale=0.35]{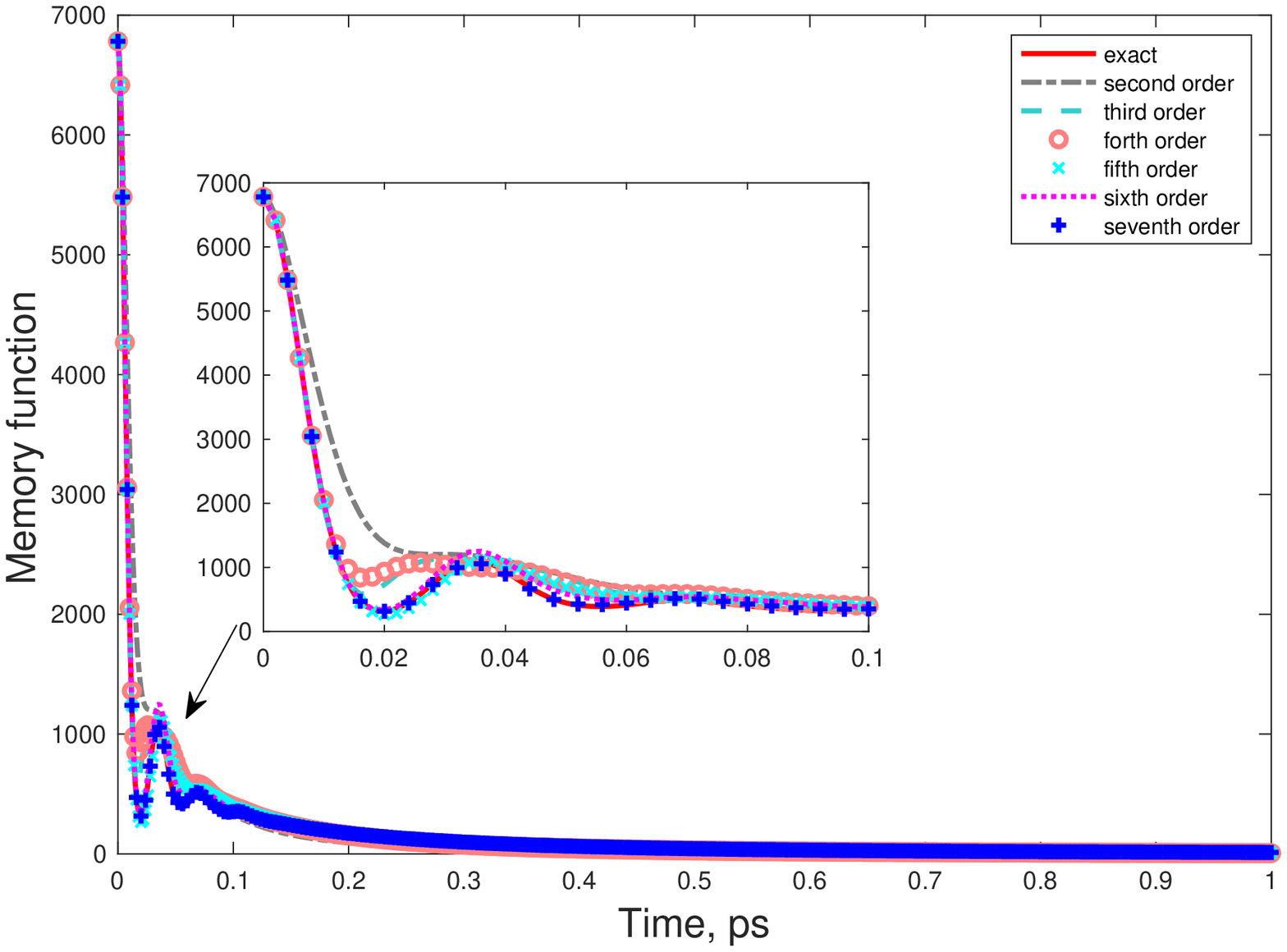}
\includegraphics[scale=0.35]{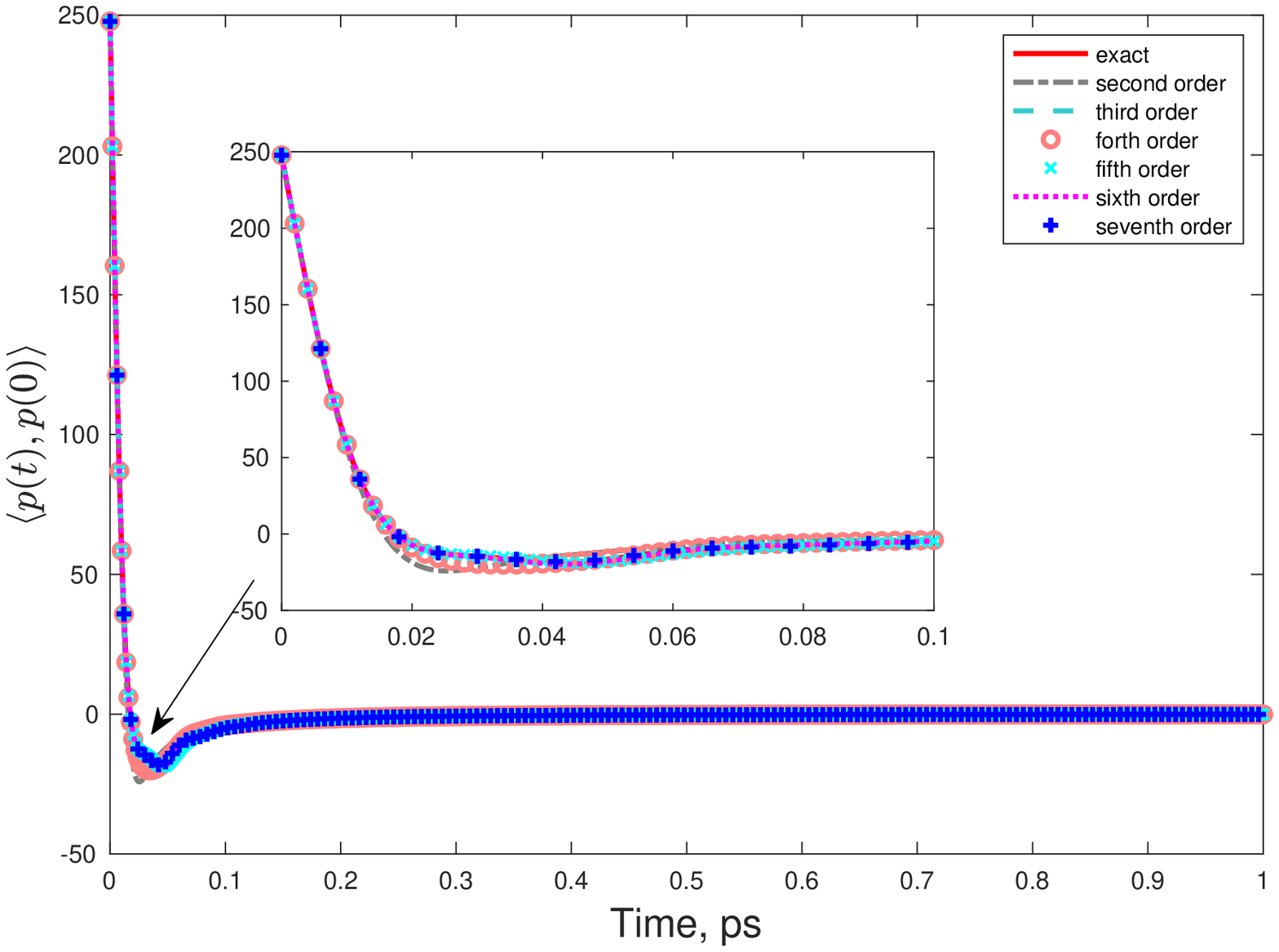}
\caption{Numerical result for $\gamma = 91$, from  second order approximation to seventh order approximation, all compared to exact solution. Left: the memory kernel function. Right: velocity auto correlation.  Both plots are for the third rotation component of the first residue.}\label{Fig: g91}
\end{figure}

In Figure \ref{Fig: g5}, we present a comparison for $\gamma = 5$. 
The small damping constant leads to a underdamped system, making  the approximation difficult due to the rapid and non-trivial oscillation. However we can observe substantial improvement of  the accuracy on the memory kernel. The memory effect on auto correlation is evident compared to system with high damping constant. 
Though improvement is significant for the memory kernel,  the velocity time correlation exhibits noticeable error. 
\begin{figure}[htp]
\includegraphics[scale=0.35]{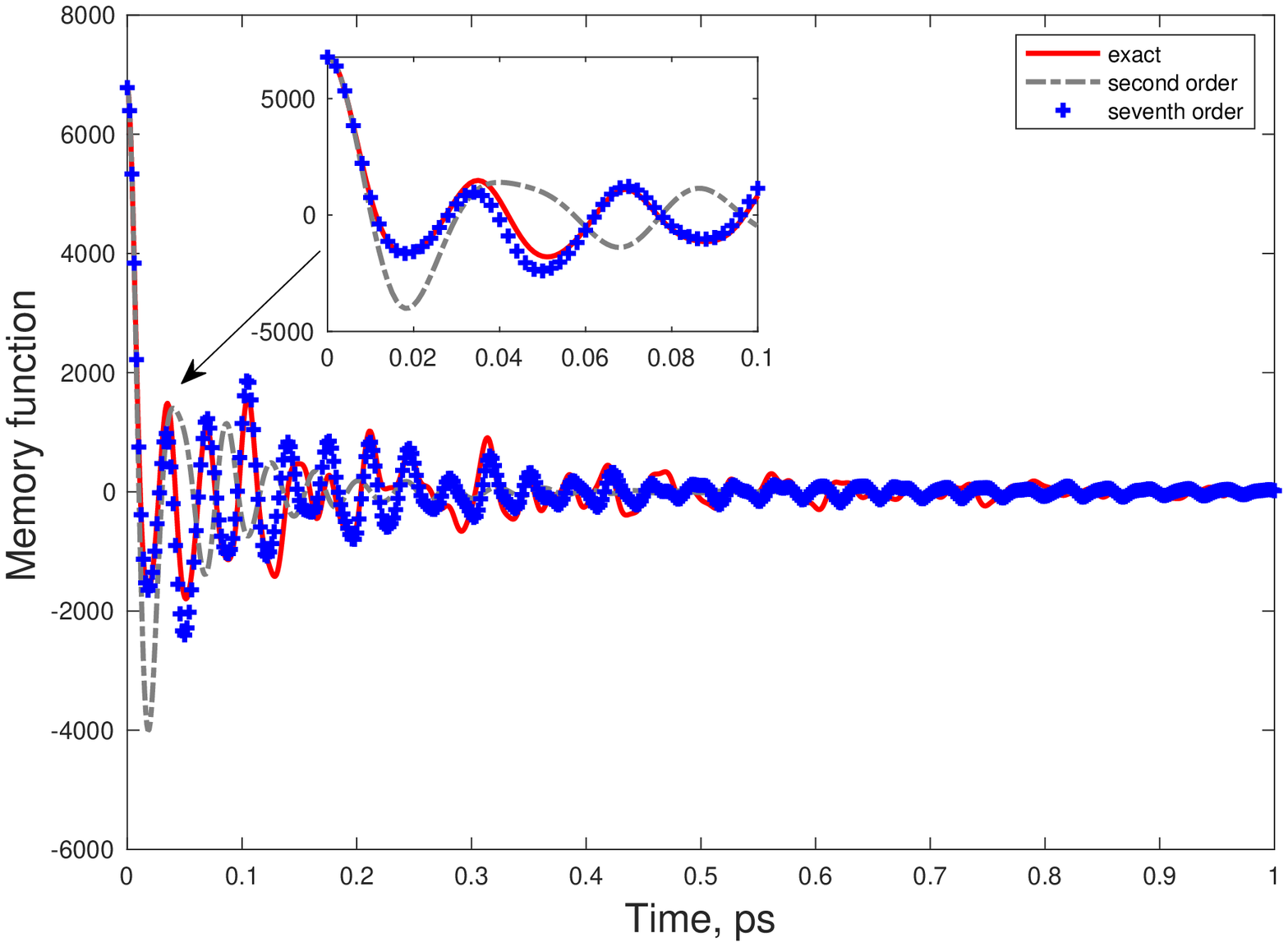}
\includegraphics[scale=0.35]{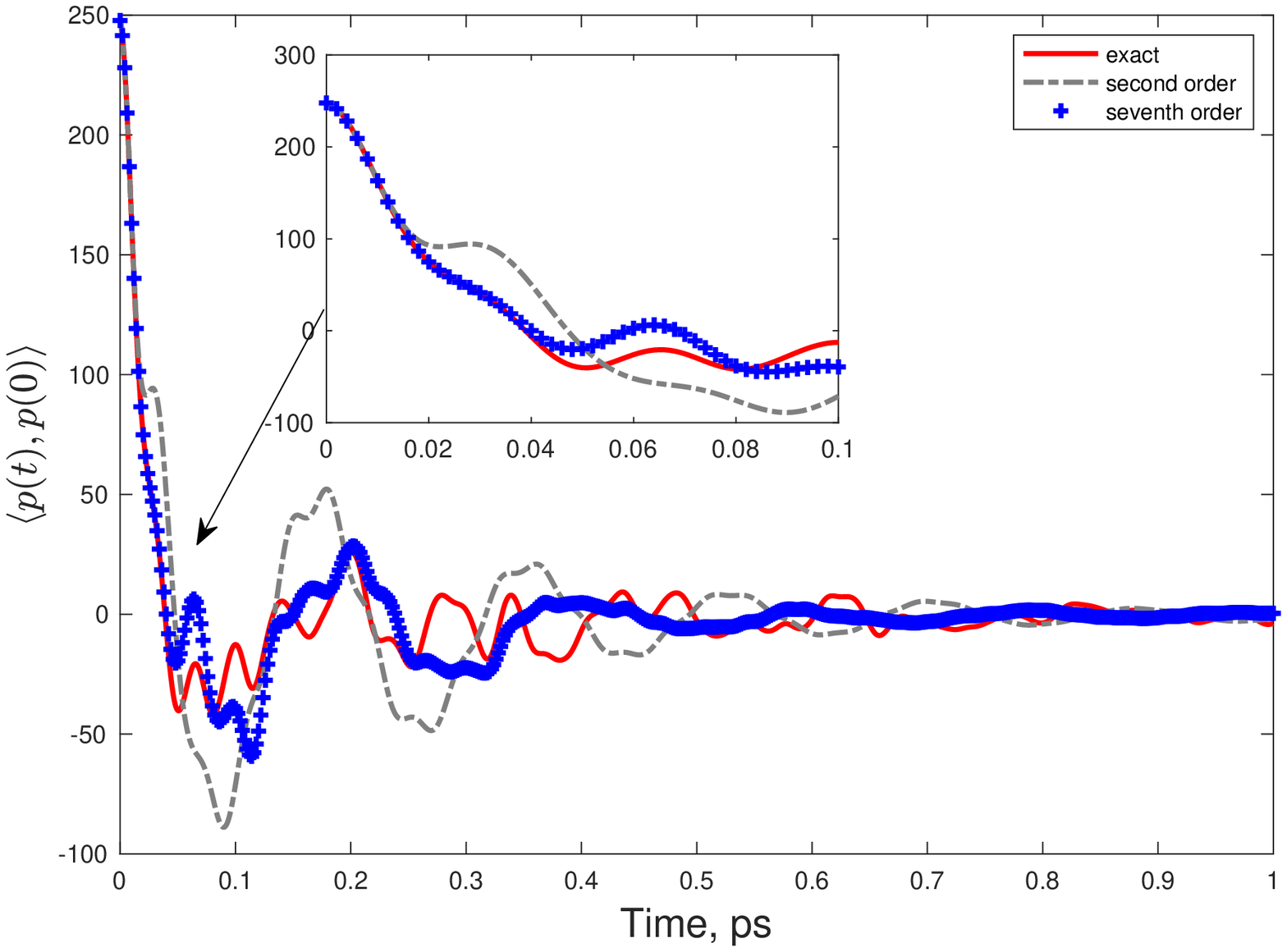}
\caption{Numerical result for $\gamma=5$. Figures show the comparison of the exact solution, second order approximation and seventh order approximation. Left: the memory kernel function. Right: velocity auto correlation. Both plots are for the third rotation component of the first residue.}\label{Fig: g5}
\end{figure}


In Figure \ref{fig: comp}, we provide a close-up view over the time interval [0,0.2] $ps$, and show results from secon order to seventh order approximations.  We observe  increased accuracy as the order of the approximation is increased within this time period. 

\begin{figure}[htp]
\includegraphics[scale=0.35]{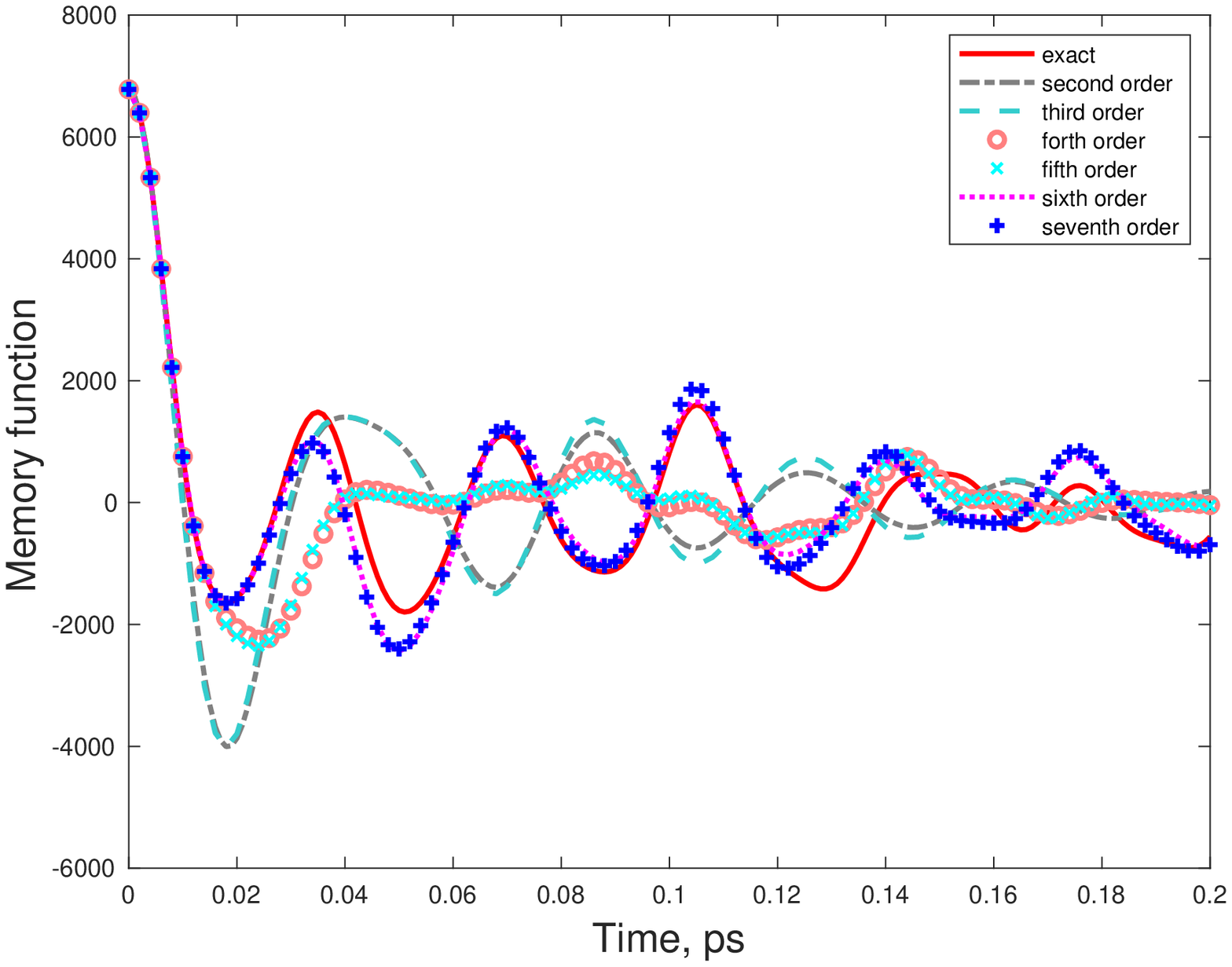}
\includegraphics[scale=0.35]{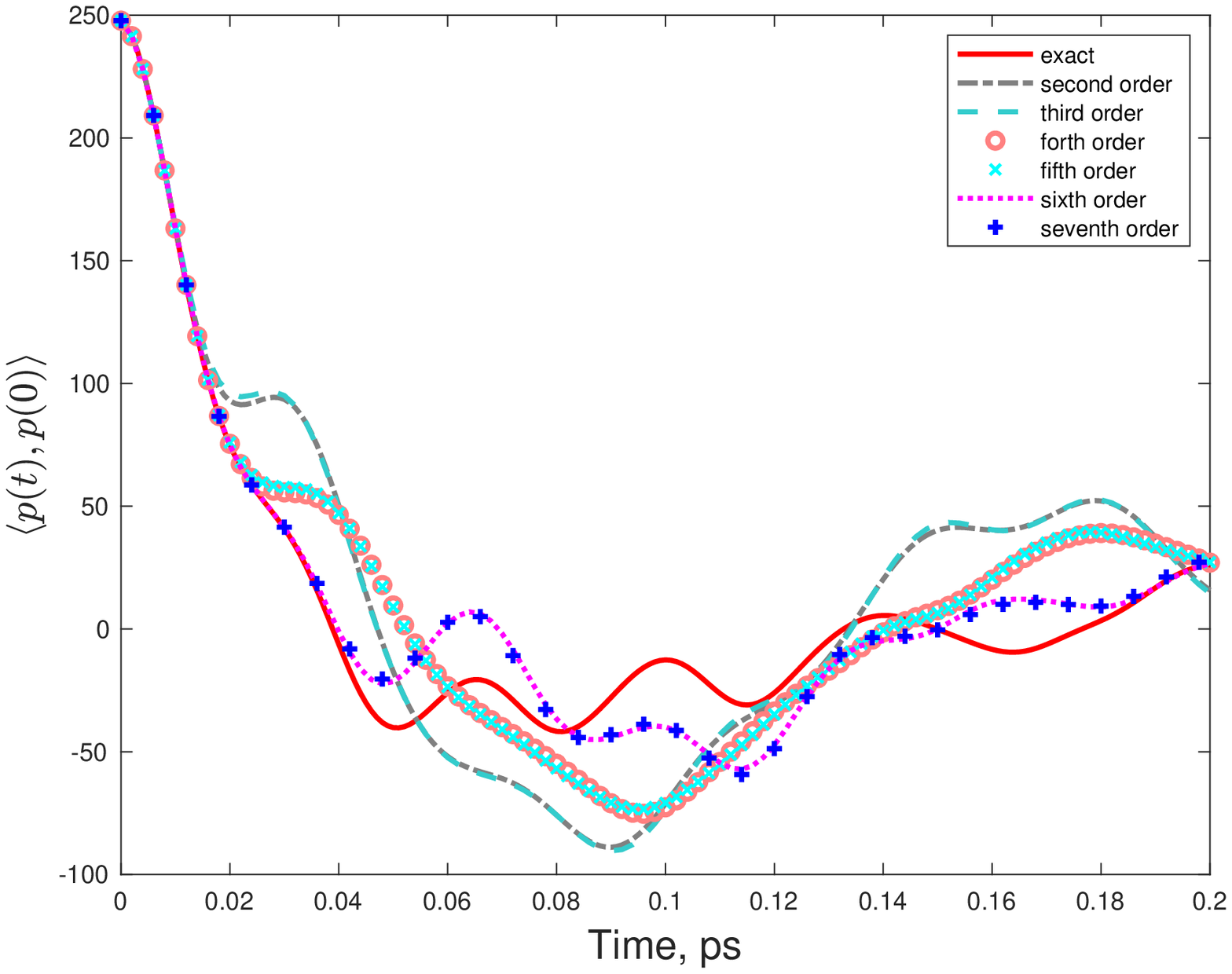}
\caption{Numerical result for $\gamma=5$. Comparison of the second order through seventh order approximations. Left: the memory kernel function. Right: velocity  auto correlation.  Both plots are for the third rotation component of the first residue.}\label{fig: comp}
\end{figure}

Finally, we present  the relative $L_2$ error for both memory kernel and time correlation, comparing the results of second order and seventh order, for both $\gamma=91$ and $\gamma=5$ in Figure \ref{fig: errL2}. This relative $L_2$ error is  computed for the time period [0,1]. We showed error for each coarse grained variables, and improvement of accuracy is significant. 
\begin{figure}[htp]
\begin{center}
\includegraphics[scale=0.38]{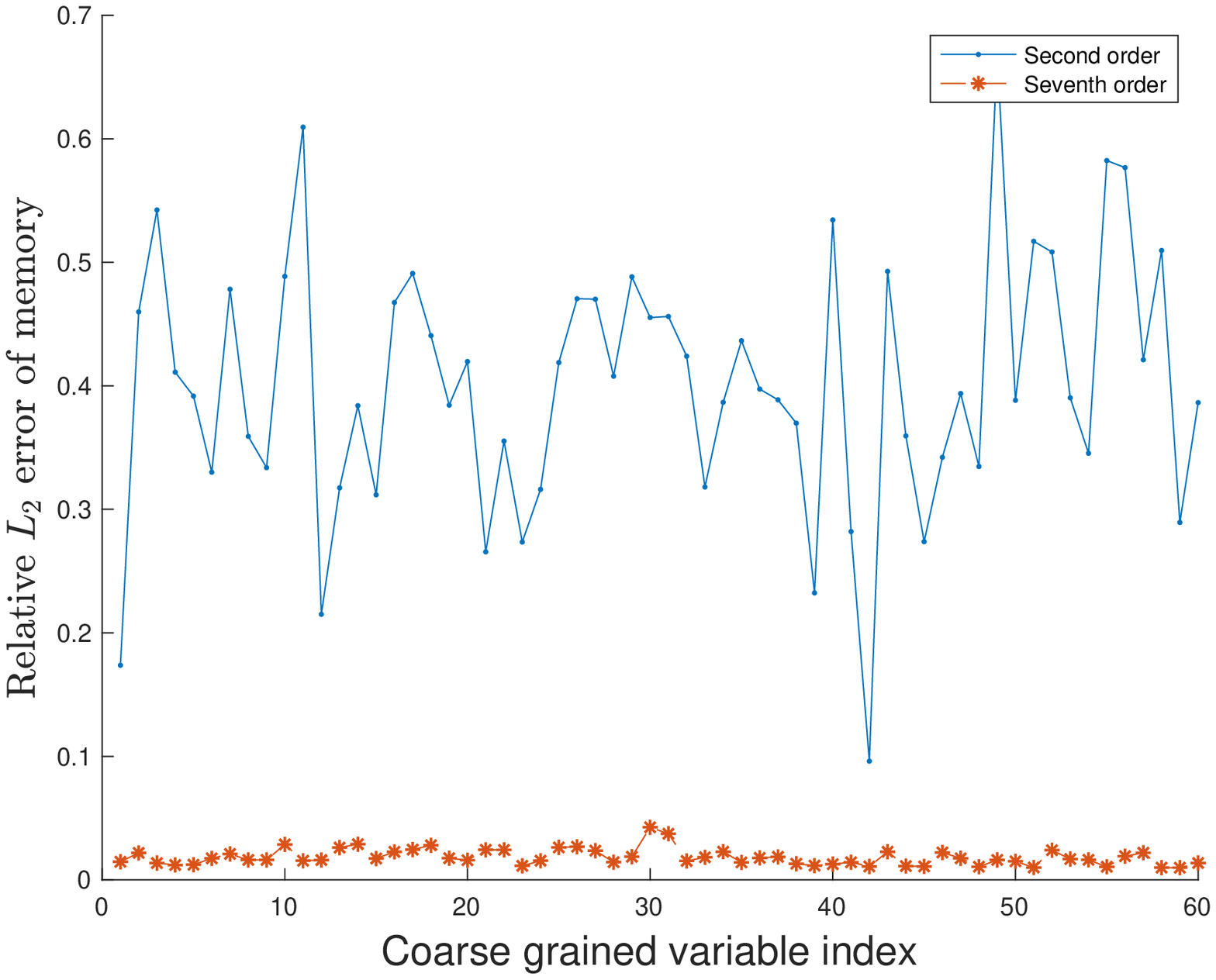}
\includegraphics[scale=0.38]{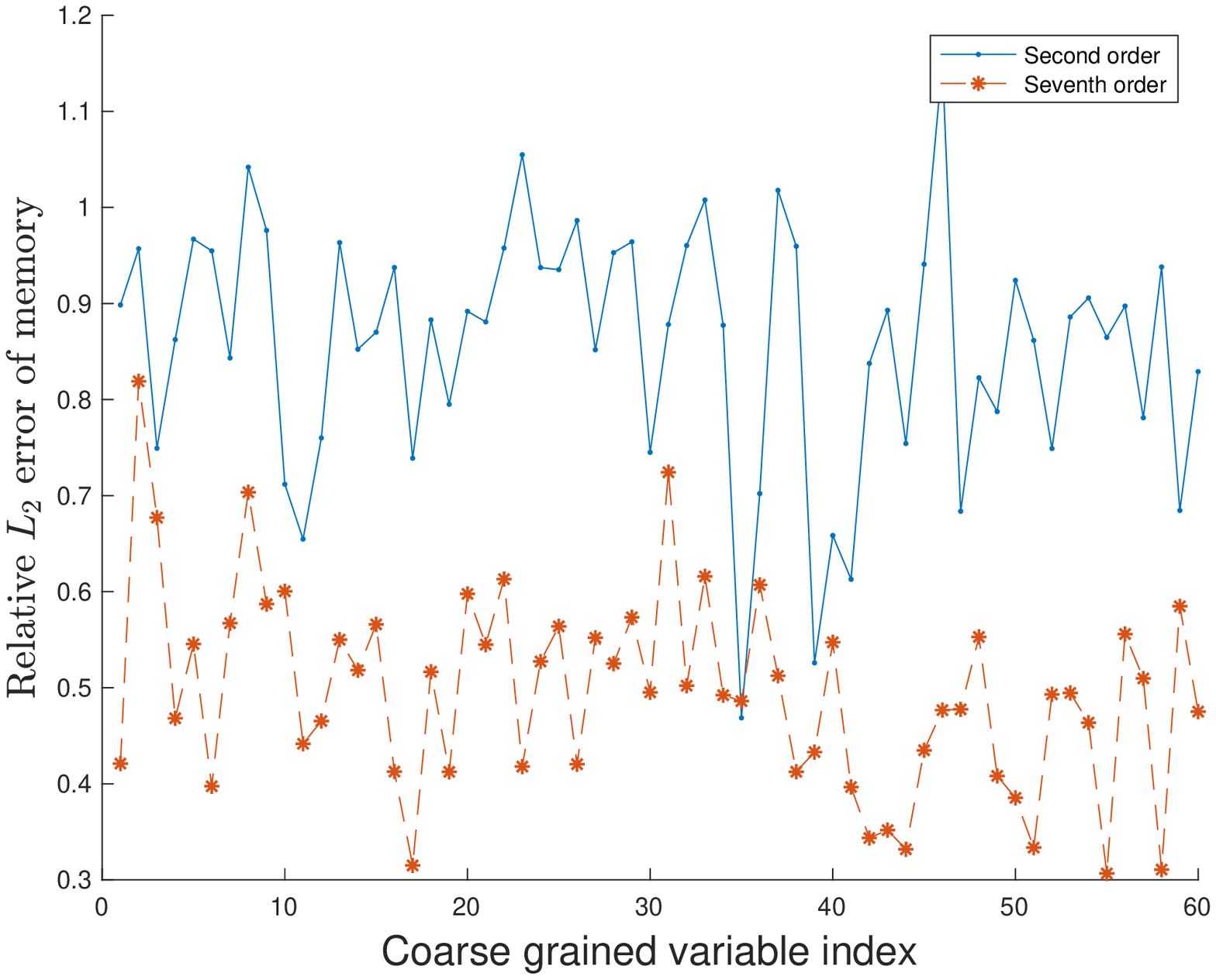}
\includegraphics[scale=0.38]{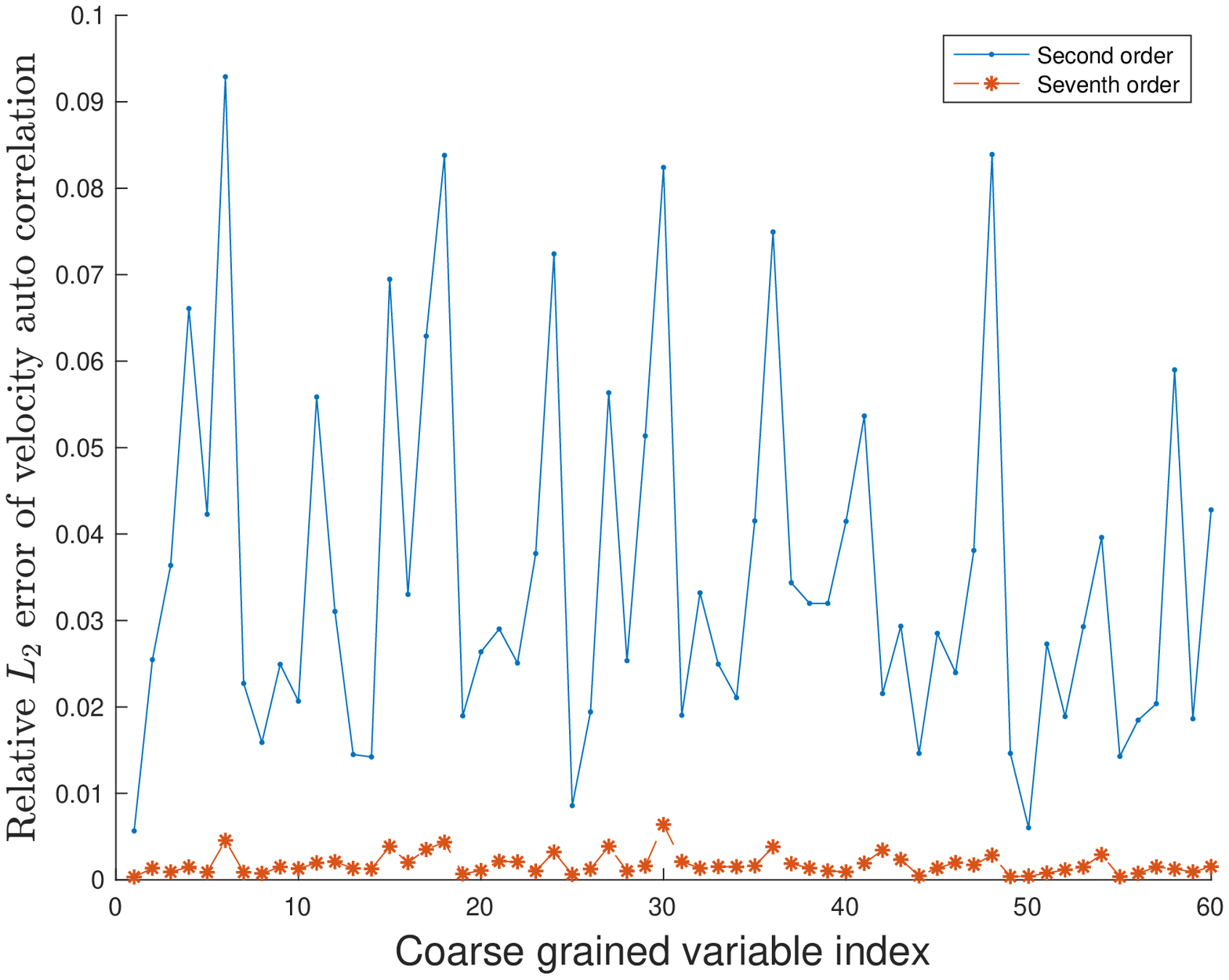}
\includegraphics[scale=0.38]{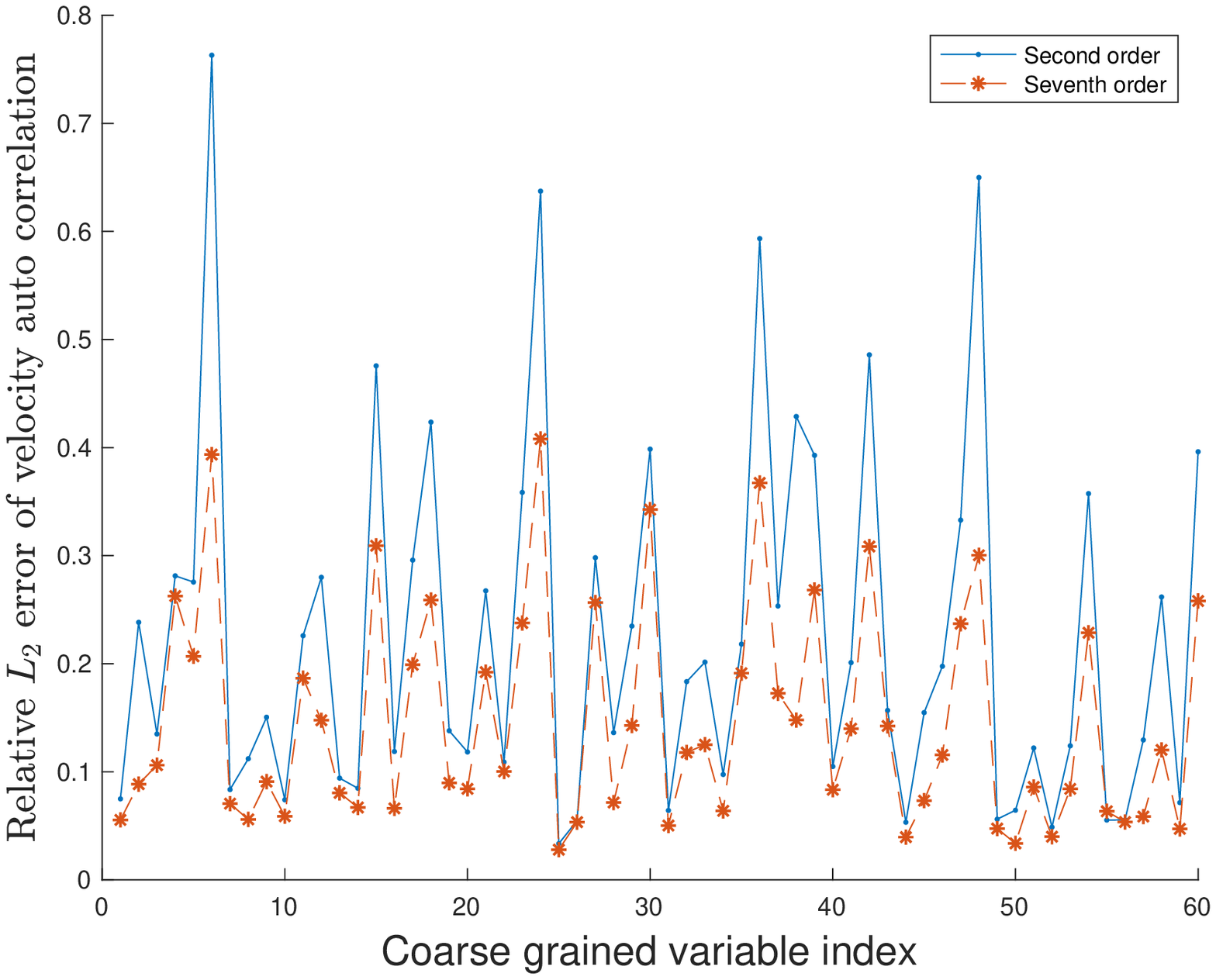}
\caption{Comparison between second and seventh order projection, using the relative $L_2$ error for each coarse grained variable. Top: the memory kernel function. Bottom: the time correlation. Left: $\gamma=91$. Right: $\gamma=5$.}\label{fig: errL2}
\end{center}
\end{figure}
}

{\blue In addition to the Krylov subspaces that were presented in the previous section, we also implemented inverse Krylov subspaces and shifted-inverse Krylov subspaces in the Galerkin projection. These variations can often offer better approximations to the transfer function in order-reduction problems \cite{bai2002krylov}, which in our case, corresponds to the memory function. However, through our numerical computations, we found that none of these choices satisfies \eqref{eq: th2}. This implies that the second FDT is not fulfilled, and the reduced dynamics \eqref{eq: Krylov'2} does {\it not} produce stationary processes \cite{Kubo66,pavliotis2014stochastic}. In fact, the variance of the solution will follow the dynamic Lyapunov equation (\cite{pavliotis2014stochastic} Eqn 3.103), and it will not converge to the steady-state.  }

\section{Conclusion}
We adopted reduced-order modeling techniques to reduce the Langevin dynamics model.  We consider reduced models obtained from Petrov-Galerkin projections. By selecting appropriate Krylov spaces,  we  show the mathematical equivalence of the proposed model to the reduced models derived from moment matching procedure. Another emphasis is placed on the statistical consistency, i.e., the fluctuation-dissipation theorem. We are able to identify two conditions that ensure such consistency. We also showed that the Galerkin projections to the selected subspaces automatically satisfy the FDT, at least for $n\le 5.$ With the block Lanczos algorithm, the models derived this way are more robust. 

One open issue is the case when the damping coefficient $\Gamma$ is not proportional to an identity matrix. In this case, both condition {\bf A} and condition {\bf B} are still sufficient to ensure the FDT. But the Krylov subspaces construction in section 4 may not satisfy condition {\bf B}. Another open question is whether one can bypass the linear approximations used in \eqref{eq: lg'} and \eqref{eq: lg'2}. It seems that a different methodology is needed to derive the generalized Langevin equation \eqref{eq: GLE}. These issues will be addressed in future works.

 \section*{Acknowledgement}
 The research of Li is supported by  NSF Grant DMS-1522617 and DMS-1619661.  The research of Liu is supported by NSF grant NSF Grant DMS-1759535 and DMS-1759536.

\appendix
\section{Recurrence Formula for $\wt {\Sigma}=W^\intercal \Sigma W$}
This is the proof of lemma 3. 
For the $n$th approximation, 
\begin{align*}
\wt \Sigma = W^\intercal \Sigma W = \left[
\begin{array}{ccccc}
LD^{-1}\Sigma D^{-\intercal}L^\intercal & LD^{-1}\Sigma L^\intercal & LD^{-1}\Sigma D^{\intercal} L^\intercal & \dots&LD^{-1}\Sigma (D^{n-2})^\intercal L^\intercal\\
L\Sigma D^{-\intercal} L^\intercal &L\Sigma L^\intercal&\dots&&L\Sigma(D^{n-2})^\intercal L^\intercal\\
\dots\\
LD^{n-2} \Sigma D^{-T} L^\intercal&\dots&&&LD^{n-2} \Sigma (D^{n-2})^\intercal L^\intercal
\end{array}
\right]
\end{align*}

The block element of $\wt {\Sigma}$ on the $i$th row and $j$th column is given by  
$$
\wt {\Sigma}_{ij} = LD^{i-2} \Sigma (D^\intercal)^{j-2} L^\intercal.
$$
By using equation (\ref{eq: Lya}), we arrive at,
$$
\wt {\Sigma}_{ij} = -k_BT [LD^{i-1} Q (D^\intercal)^{j-2} L^\intercal+ LD^{i-2}Q(D^\intercal)^{j-1} L^\intercal].
$$

Next we define matrix $S=\left[\begin{array}{cc}A_{22}^{-1} &0\\0&-I\end{array}\right]$. It can be easily seen that since $\Gamma=\gamma I$, we have $$SD^\intercal = DS, \quad SL^\intercal=R, \quad Q-S=\frac{1}{2\gamma k_BT}\Sigma.$$

With these identities, we are able to manipulate terms, and get,
\begin{align*}
\wt {\Sigma}_{ij} &= -k_BT LD^{i-1} S (D^\intercal)^{j-2}L^\intercal - k_BT LD^{i-2} S (D^\intercal)^{j-1}L^\intercal -  \frac{1}{\gamma} LD^{i-1} \Sigma (D^\intercal)^{j-2}L^\intercal- \frac{1}{\gamma}LD^{i-1} \Sigma (D^\intercal)^{j-2}L^\intercal\\
&= -2 k_BT LD^{i+j-3} R - \frac{1}{\gamma}\wt {\Sigma}_{i+1,j} - \frac{1}{\gamma}\wt {\Sigma}_{i,j+1} =  - \frac{1}{\gamma} \wt  \Sigma_{i+1,j} - \frac {1}{\gamma} \wt  \Sigma_{i,j+1} - 2 k_BTM_{i+j-3},
\end{align*}
where we have used the notation $M_{-1}=-M_\infty$.

Meanwhile, the block elements of $\wt \Sigma$ in the second column and the second row are all zeros, since by direct calculation, $\Sigma L^\intercal=0$.  Furthermore, we have $\wt  \Sigma_{11}=2k_BTM_{\infty}$. For example when $j=1$, we have
$$
\wt  \Sigma_{i1} = -\gamma \wt \Sigma_{i-1,j}-2k_BT M_{i-3} \quad, i\geq 3.
$$
Therefore we are able to write out entries of $\wt \Sigma$ using $M_i$s. 

\section{Representation of $M_i$s}
This is the proof of lemma 4.
This calculation is based on the formulas in 
\eqref{eq: mem} and \eqref{eq: Th} for the memory function $\theta(t).$ The Laplace transform will be given by,
\begin{equation}
  \Theta = L \big[ \lambda^{-1} I - D\big] R = 
  \lambda L \left[
  \begin{array}{cc}
    I & -\lambda I \\
    \lambda A_{22} & (1+\lambda \gamma)I
    \end{array} \right]^{-1}  R = \lambda L 
    \left[
  \begin{array}{cc}
    ( I + \lambda^2 (1+\lambda \gamma)^{-1}A_{22})^{-1} & * \\
   * & *
    \end{array} \right] R. 
\end{equation}
Here we have used a block inversion formula, and the fact that the second block in both $L$ and $R$ are zero. 

At this point, we can invoke the Neumann series of the first diagonal block and we have,
{
\begin{equation}
 \Theta = \lambda A_{12}A^{-1}_{22}A_{21} -  \lambda^3 A_{12}  (1 -\lambda \gamma + \lambda^2\gamma^2 + \cdots) A_{21} + 
  -  \lambda^5 A_{12}  (1 -\lambda \gamma + \lambda^2\gamma^2 + \cdots)^2 A_{22}A_{21} +
 \cdots.
\end{equation}
}
Therefore, the patterns in the representation of $M_i$'s can be observed.

As an example, the first few moments are listed below
\begin{align*}
&M_{\infty}=\gamma A_{12}(A_{22}^{-1})^2A_{21}, M_0=A_{12}A^{-1}_{22}A_{21}, \quad M_1=0, \\
&M_2=-A_{12}A_{21},\quad M_3=\gamma A_{12}A_{21},\quad M_4=A_{12}A_{22}A_{21}-\gamma^2 A_{12}A_{21},\\
&M_{5}=-2\gamma A_{12}A_{22}A_{21}+\gamma^3A_{12}A_{21},\quad
M_{6} = -A_{12}A_{22}A_{21}+3\gamma^2 A_{12}A_{22}A_{21}-\gamma^4 A_{12}A_{21},\\
&M_7 = 3\gamma A_{12}A_{22}^2A_{21} - 4\gamma^3 A_{12}A_{22}A_{21} + \gamma^5 A_{12}A_{21}.
\end{align*}

In addition, here are a few identities that is used to prove second FDT for order $n=3, 4, 5$. 
$$
\gamma M_2+M_3=0,
\quad 
\gamma^2 M_3 + 2 \gamma M_4 +M_5 = 0,\quad
\gamma^3 M_4 +3 \gamma^2 M_5 +3\gamma M_6 + M_7 =0.
$$

\bibliographystyle{unsrt}
\bibliography{mem0,GLERef,GLERef1,reduction,add-ref}

 \end{document}